\documentclass{amsart} 

\usepackage{amsmath, amssymb, amsfonts, tikz, hyperref, comment, url, todonotes}
\usetikzlibrary{math}
\usetikzlibrary{decorations.pathreplacing}

\newtheorem{theorem}{Theorem}[section]
\newtheorem{lemma}[theorem]{Lemma}

\theoremstyle{definition}
\newtheorem{example}[theorem]{Example}
\newtheorem{remark}[theorem]{Remark}
\newtheorem{definition}[theorem]{Definition}
\newtheorem*{definition*}{Definition}
\newtheorem*{theorem*}{Theorem}
\newtheorem*{proposition*}{Proposition}
\newtheorem*{corollary*}{Corollary}

\newtheorem{conjecture}[theorem]{Conjecture}

\newcommand{\Z}{\mathbb{Z}}
\renewcommand{\S}{\mathfrak{S}}
\newcommand{\wt}{\mathrm{wt}}
\newcommand{\RC}{\mathcal{RC}}
\newcommand{\code}{\mathrm{code}}
\renewcommand{\P}{\mathcal{P}}
\newcommand{\diag}{\mathrm{diag}}
\newcommand{\id}{\mathrm{id}}

\author{Yibo Gao}
\address{Department of Mathematics, Massachusetts Institute of Technology, \mbox{Cambridge, MA 02139}}
\email{\href{mailto:gaoyibo@mit.edu}{{\tt gaoyibo@mit.edu}}}

\begin{document}
\title{Principal specializations of Schubert polynomials and pattern containment}
\date{\today}

\begin{abstract}
We show that the principal specialization of the Schubert polynomial at $w$ is bounded below by $1+p_{132}(w)+p_{1432}(w)$ where $p_u(w)$ is the number of occurrences of the pattern $u$ in $w$, strengthening a previous result by A. Weigandt. We then make a conjecture relating the principal specialization of Schubert polynomials to pattern containment. Finally, we characterize permutations $w$ whose RC-graphs are connected by simple ladder moves via pattern avoidance. 
\end{abstract}
\maketitle

\section{Introduction}\label{sec:intro}
The study of principal specializations of Schubert polynomials $\nu_w:=\S_w(1,\ldots,1)$ has seen interesting results in recent years. Geometrically, $\nu_w$ equals the degree of the matrix Schubert variety \cite{knutson2005grobner} corresponding to $w$ and combinatorially, $\nu_w$ equals the number of RC-graphs of $w$. For a fixed $n\in\Z_{>0}$, Stanley \cite{stanley2017some} asked the question of determining the permutations $w\in S_n$ that achieve the maximum value of $\nu_w$ and it is conjectured by Merzon and Smirnov \cite{merzon2016determinantal} that such $w$ must be layered. Assuming that $w$ is layered, Morales, Pak and Panova \cite{morales2019asymptotics} determined such $w$'s that achieve the maximum value of $\nu_w$ and the asymptotic behavior of this value.

\begin{definition}\label{def:patternContainment}
For $u\in S_k$ and $w\in S_n$, define the number of \textit{occurrences} of $u$ in $w$ to be
\begin{align*}
p_{u}(w):=\#\{1\leq a_1<\cdots<a_k\leq n\ |\ &w(a_i)<w(a_j)\text{ if and only if }u(i)<u(j),\\
&\text{ for all }1\leq i<j\leq k\}.
\end{align*}
\end{definition}
We can extend Definition~\ref{def:patternContainment} to $w\in S_{\infty}$, permutations of $\Z_{>0}$ with all but a finite number of fixed points, as long as $u(k)\neq k$.

Weigandt \cite{weigandt2018schubert} showed that $\nu_w$ is bounded below by $1+p_{132}(w)$, where $p_{132}(w)$ is the number of 132 patterns in $w$. However, this inequality is far from being tight. In this paper, we strengthen this inequality (Section~\ref{sec:main}) and provide a conjecture (Section~\ref{sec:conj}) that strongly relates the principal specializations of Schubert polynomials $\nu_w$ to enumeration of pattern containment in $w$.

We start with some background on Schubert polynomials and RC-graphs. We refer readers to \cite{bergeron1993RC} and \cite{manivel2001symmetric} for detailed exposition on this subject matter. 

An \textit{RC-graph} of $w$ is a finite subset $D\subset\Z_{>0}\times\Z_{>0}$ such that $$\prod_{i=1}^{\infty}\prod_{j=\infty,(i,j)\in D}^{1}s_{i+j-1}=w$$
is a reduced word for $w$. The set of RC-graphs of $w$ is denoted as $\RC(w)$. For an RC-graph $D$, its \textit{weight} is defined to be $\wt(D)=\prod_{(i,j)\in D}x^i$. The following theorem is well-known.
\begin{theorem}[\cite{bergeron1993RC}]
For a permutation $w\in S_{\infty}$, $\S_w=\sum_{D\in\RC(w)}\wt(D).$
\end{theorem}
As a result, the principal specialization of Schubert polynomials $\nu_w:=\S_w(1,\ldots,1)$ equals $\#\RC(w)$, the number of RC-graphs for $w$, which is the quantity that we are interested in.

An RC-graph $D$ may be viewed as a \textit{strand diagram} in the 2D grid $\Z_{>0}^2$ by placing a \textit{crossing} 
\begin{tikzpicture}[scale=0.15]
\draw (-1,0)--(1,0);
\draw (0,-1)--(0,1);
\end{tikzpicture}
at each $(i,j)\in D$ and placing an \textit{elbow}
\begin{tikzpicture}[scale=0.15]
\draw (-1,0) arc (270:360:1);
\draw (1,0) arc (90:180:1);
\end{tikzpicture}
(or \textit{non-crossing}) at each $(i,j)\notin D$. This procedure produces a pseudo-line arrangement that connects $(k,0)$ with $(0,w(k))$ where no two strands intersect more than once where $D\in\RC(w)$. For a clearer view of the local moves, we are also going to denote a crossing by $+$ and an elbow by $\cdot$. See Figure~\ref{fig:RCexample} for an example.

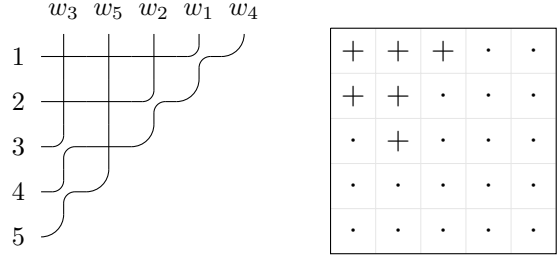
\begin{figure}[h!]
\begin{tikzpicture}[scale=0.300000000000000]
\draw(2,-3)--(4,-3);
\draw(3,-2)--(3,-4);
\draw(4,-3)--(6,-3);
\draw(5,-2)--(5,-4);
\draw(6,-3)--(8,-3);
\draw(7,-2)--(7,-4);
\draw(9,-2)--(9,-2.50000000000000);
\draw(9,-3.50000000000000)--(9,-4);
\draw(8,-3)--(8.50000000000000,-3);
\draw(9.50000000000000,-3)--(10,-3);
\draw(8.50000000000000,-3)arc(270:360:0.500000000000000);
\draw(9.50000000000000,-3)arc(90:180:0.500000000000000);
\draw(2,-5)--(4,-5);
\draw(3,-4)--(3,-6);
\draw(4,-5)--(6,-5);
\draw(5,-4)--(5,-6);
\draw(7,-4)--(7,-4.50000000000000);
\draw(7,-5.50000000000000)--(7,-6);
\draw(6,-5)--(6.50000000000000,-5);
\draw(7.50000000000000,-5)--(8,-5);
\draw(6.50000000000000,-5)arc(270:360:0.500000000000000);
\draw(7.50000000000000,-5)arc(90:180:0.500000000000000);
\draw(3,-6)--(3,-6.50000000000000);
\draw(3,-7.50000000000000)--(3,-8);
\draw(2,-7)--(2.50000000000000,-7);
\draw(3.50000000000000,-7)--(4,-7);
\draw(2.50000000000000,-7)arc(270:360:0.500000000000000);
\draw(3.50000000000000,-7)arc(90:180:0.500000000000000);
\draw(4,-7)--(6,-7);
\draw(5,-6)--(5,-8);
\draw(3,-8)--(3,-8.50000000000000);
\draw(3,-9.50000000000000)--(3,-10);
\draw(2,-9)--(2.50000000000000,-9);
\draw(3.50000000000000,-9)--(4,-9);
\draw(2.50000000000000,-9)arc(270:360:0.500000000000000);
\draw(3.50000000000000,-9)arc(90:180:0.500000000000000);
\draw(10,-3)arc(270:360:1);
\draw(8,-5)arc(270:360:1);
\draw(6,-7)arc(270:360:1);
\draw(4,-9)arc(270:360:1);
\draw(2,-11)arc(270:360:1);
\node at (1,-3) {$1$};
\node at (9,-1) {$w_{1}$};
\node at (1,-5) {$2$};
\node at (7,-1) {$w_{2}$};
\node at (1,-7) {$3$};
\node at (3,-1) {$w_{3}$};
\node at (1,-9) {$4$};
\node at (11,-1) {$w_{4}$};
\node at (1,-11) {$5$};
\node at (5,-1) {$w_{5}$};
\end{tikzpicture}
\qquad
\begin{tikzpicture}[scale=0.300000000000000]
\node at (3,-3) {\Large $+$};
\node at (5,-3) {\Large $+$};
\node at (7,-3) {\Large $+$};
\node at (9,-3) {\Large $\cdot$};
\node at (11,-3) {\Large $\cdot$};
\node at (3,-5) {\Large $+$};
\node at (5,-5) {\Large $+$};
\node at (7,-5) {\Large $\cdot$};
\node at (9,-5) {\Large $\cdot$};
\node at (11,-5) {\Large $\cdot$};
\node at (3,-7) {\Large $\cdot$};
\node at (5,-7) {\Large $+$};
\node at (7,-7) {\Large $\cdot$};
\node at (9,-7) {\Large $\cdot$};
\node at (11,-7) {\Large $\cdot$};
\node at (3,-9) {\Large $\cdot$};
\node at (5,-9) {\Large $\cdot$};
\node at (7,-9) {\Large $\cdot$};
\node at (9,-9) {\Large $\cdot$};
\node at (11,-9) {\Large $\cdot$};
\node at (3,-11) {\Large $\cdot$};
\node at (5,-11) {\Large $\cdot$};
\node at (7,-11) {\Large $\cdot$};
\node at (9,-11) {\Large $\cdot$};
\node at (11,-11) {\Large $\cdot$};
\draw[gray!20!white,ultra thin,step=2] (2,-2) grid (12,-12);
\draw(2,-2)--(12,-2)--(12,-12)--(2,-12)--(2,-2);
\end{tikzpicture}
\caption{An RC-graph for $w=43152$}
\label{fig:RCexample}
\end{figure}
We label the strands such that strand $k$ starts from the left of row $k$ and ends at the top of column $w(k)$. For a crossing $(i,j)\in D$, we say that it has \textit{type} $(a,b)$ if the two strands intersecting at $(i,j)$ are labeled with $a<b$. It is clear that all crossings of $D\in\RC(w)$ have different types and are in one-to-one correspondence with $\{(i,j)\ |\ i<j,\ w(i)>w(j)\}$, the inversions of $w$. 

There are distinguished RC-graphs of $w$. Recall that the \textit{Rothe diagram} of $w$ is the set
$$RD(w):=\{(i,j)\ |\ 1\leq i,j\leq n,\ w(i)>j,\ w^{-1}(j)>i\}.$$
The \textit{bottom RC-graph} of $w$ is the set of squares left-justified from $RD(w)$, which formally is $B_w:=\{(i,j)\ |\ j\leq\code(w)_i\}$ where $\code(w)$ is the \textit{Lehmar code} of $w$ defined via $\code(w)_i=\#\{j>i\ |\ w(j)<w(i)\}.$ Similarly, the \textit{top RC-graph} of $w$ is the set of sqaures top-adjusted from $RD(w)$, which is $T_w:=\{(i,j)\ |\ i\leq\code(w^{-1})_j\}$. An example is shown in Figure~\ref{fig:rotheexample}.
\begin{figure}[h!]
\centering
\begin{tikzpicture}[scale=0.300000000000000]
\draw(2,-2)--(14,-2)--(14,-14)--(2,-14)--(2,-2);
\draw[gray!20!white,ultra thin,step=2] (2,-2) grid (14,-14);
\node at (3,-3) {$\bullet$};
\draw(3,-14)--(3,-3)--(14,-3);
\node at (11,-5) {$\bullet$};
\draw(11,-14)--(11,-5)--(14,-5);
\node at (13,-7) {$\bullet$};
\draw(13,-14)--(13,-7)--(14,-7);
\node at (7,-9) {$\bullet$};
\draw(7,-14)--(7,-9)--(14,-9);
\node at (9,-11) {$\bullet$};
\draw(9,-14)--(9,-11)--(14,-11);
\node at (5,-13) {$\bullet$};
\draw(5,-14)--(5,-13)--(14,-13);
\fill[gray!50!white, draw=black] (4,-4) rectangle (6,-6);
\fill[gray!50!white, draw=black] (6,-4) rectangle (8,-6);
\fill[gray!50!white, draw=black] (8,-4) rectangle (10,-6);
\fill[gray!50!white, draw=black] (4,-6) rectangle (6,-8);
\fill[gray!50!white, draw=black] (6,-6) rectangle (8,-8);
\fill[gray!50!white, draw=black] (8,-6) rectangle (10,-8);
\fill[gray!50!white, draw=black] (4,-8) rectangle (6,-10);
\fill[gray!50!white, draw=black] (4,-10) rectangle (6,-12);
\end{tikzpicture}
\quad
\begin{tikzpicture}[scale=0.300000000000000]
\node at (3,-3) {\Large $\cdot$};
\node at (5,-3) {\Large $\cdot$};
\node at (7,-3) {\Large $\cdot$};
\node at (9,-3) {\Large $\cdot$};
\node at (11,-3) {\Large $\cdot$};
\node at (13,-3) {\Large $\cdot$};
\node at (3,-5) {\Large $+$};
\node at (5,-5) {\Large $+$};
\node at (7,-5) {\Large $+$};
\node at (9,-5) {\Large $\cdot$};
\node at (11,-5) {\Large $\cdot$};
\node at (13,-5) {\Large $\cdot$};
\node at (3,-7) {\Large $+$};
\node at (5,-7) {\Large $+$};
\node at (7,-7) {\Large $+$};
\node at (9,-7) {\Large $\cdot$};
\node at (11,-7) {\Large $\cdot$};
\node at (13,-7) {\Large $\cdot$};
\node at (3,-9) {\Large $+$};
\node at (5,-9) {\Large $\cdot$};
\node at (7,-9) {\Large $\cdot$};
\node at (9,-9) {\Large $\cdot$};
\node at (11,-9) {\Large $\cdot$};
\node at (13,-9) {\Large $\cdot$};
\node at (3,-11) {\Large $+$};
\node at (5,-11) {\Large $\cdot$};
\node at (7,-11) {\Large $\cdot$};
\node at (9,-11) {\Large $\cdot$};
\node at (11,-11) {\Large $\cdot$};
\node at (13,-11) {\Large $\cdot$};
\node at (3,-13) {\Large $\cdot$};
\node at (5,-13) {\Large $\cdot$};
\node at (7,-13) {\Large $\cdot$};
\node at (9,-13) {\Large $\cdot$};
\node at (11,-13) {\Large $\cdot$};
\node at (13,-13) {\Large $\cdot$};
\draw[gray!20!white,ultra thin,step=2] (2,-2) grid (14,-14);
\draw(2,-2)--(14,-2)--(14,-14)--(2,-14)--(2,-2);
\end{tikzpicture}
\quad
\begin{tikzpicture}[scale=0.300000000000000]
\node at (3,-3) {\Large $\cdot$};
\node at (5,-3) {\Large $+$};
\node at (7,-3) {\Large $+$};
\node at (9,-3) {\Large $+$};
\node at (11,-3) {\Large $\cdot$};
\node at (13,-3) {\Large $\cdot$};
\node at (3,-5) {\Large $\cdot$};
\node at (5,-5) {\Large $+$};
\node at (7,-5) {\Large $+$};
\node at (9,-5) {\Large $+$};
\node at (11,-5) {\Large $\cdot$};
\node at (13,-5) {\Large $\cdot$};
\node at (3,-7) {\Large $\cdot$};
\node at (5,-7) {\Large $+$};
\node at (7,-7) {\Large $\cdot$};
\node at (9,-7) {\Large $\cdot$};
\node at (11,-7) {\Large $\cdot$};
\node at (13,-7) {\Large $\cdot$};
\node at (3,-9) {\Large $\cdot$};
\node at (5,-9) {\Large $+$};
\node at (7,-9) {\Large $\cdot$};
\node at (9,-9) {\Large $\cdot$};
\node at (11,-9) {\Large $\cdot$};
\node at (13,-9) {\Large $\cdot$};
\node at (3,-11) {\Large $\cdot$};
\node at (5,-11) {\Large $\cdot$};
\node at (7,-11) {\Large $\cdot$};
\node at (9,-11) {\Large $\cdot$};
\node at (11,-11) {\Large $\cdot$};
\node at (13,-11) {\Large $\cdot$};
\node at (3,-13) {\Large $\cdot$};
\node at (5,-13) {\Large $\cdot$};
\node at (7,-13) {\Large $\cdot$};
\node at (9,-13) {\Large $\cdot$};
\node at (11,-13) {\Large $\cdot$};
\node at (13,-13) {\Large $\cdot$};
\draw[gray!20!white,ultra thin,step=2] (2,-2) grid (14,-14);
\draw(2,-2)--(14,-2)--(14,-14)--(2,-14)--(2,-2);
\end{tikzpicture}
\caption{The Rothe diagram $RD(w)$, bottom RC-graph $B_w$ and top RC-graph $T_w$ for $w=156342$.}
\label{fig:rotheexample}
\end{figure}
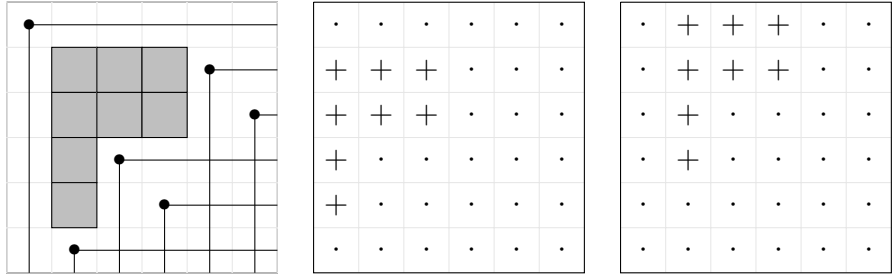

Bergeron and Billey \cite{bergeron1993RC} described the following local moves on $\RC(w)$, which are called \textit{ladder moves}.
\begin{definition}\label{def:ladder}
A \textit{ladder move} of $D\in\RC(w)$ at a crossing $(i,j)\in D$ of \textit{order} $k\geq0$ produces another RC-graph $D'=D\setminus\{(i,j)\}\cup\{(i{-}k{-}1,j{+}1)\}\in\RC(w)$ if the following conditions on $D$ are satisfied:
\begin{enumerate}
    \item $(i,j)\in D$, $(i,j{+}1),\ (i{-}k{-}1,j),\ (i{-}k{-}1,j{+}1)\notin D$;
    \item for all $i-k\leq i'<i$, $(i',j),(i',j{+}1)\in D$.
\end{enumerate}
\end{definition}
It is straightforward to observe that such a ladder move preserves the permutation and the reducedness. Visually, a ladder move is depicted in Figure~\ref{fig:ladderexample}.
\begin{figure}[h!]
\centering
\begin{tikzpicture}[scale=0.300000000000000]
\node at (3,-3) {\Large $\cdot$};
\node at (5,-3) {\Large $\cdot$};
\node at (3,-5) {\Large $+$};
\node at (5,-5) {\Large $+$};
\node at (3,-7) {\Large $\vdots$};
\node at (5,-7) {\Large $\vdots$};
\node at (3,-9) {\Large $\vdots$};
\node at (5,-9) {\Large $\vdots$};
\node at (3,-11) {\Large $+$};
\node at (5,-11) {\Large $+$};
\node at (3,-13) {\Large $+$};
\node at (5,-13) {\Large $\cdot$};
\draw[gray!20!white,ultra thin,step=2] (2,-2) grid (6,-14);
\draw(2,-2)--(6,-2)--(6,-14)--(2,-14)--(2,-2);
\draw [decorate,decoration={brace,amplitude=5pt}]
(2,-12) -- (2,-4) node [black,midway,xshift=-10pt] {$k$};
\node at (9,-8) {$\mapsto$};
\end{tikzpicture}
\quad
\begin{tikzpicture}[scale=0.300000000000000]
\node at (3,-3) {\Large $\cdot$};
\node at (5,-3) {\Large $+$};
\node at (3,-5) {\Large $+$};
\node at (5,-5) {\Large $+$};
\node at (3,-7) {\Large $\vdots$};
\node at (5,-7) {\Large $\vdots$};
\node at (3,-9) {\Large $\vdots$};
\node at (5,-9) {\Large $\vdots$};
\node at (3,-11) {\Large $+$};
\node at (5,-11) {\Large $+$};
\node at (3,-13) {\Large $\cdot$};
\node at (5,-13) {\Large $\cdot$};
\draw[gray!20!white,ultra thin,step=2] (2,-2) grid (6,-14);
\draw(2,-2)--(6,-2)--(6,-14)--(2,-14)--(2,-2);
\end{tikzpicture}
\caption{A ladder move of order $k$.}
\label{fig:ladderexample}
\end{figure}
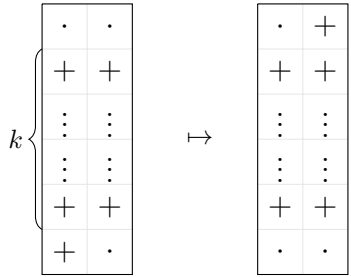

Let $D\mapsto D'$ be a ladder move that removes $(i,j)\in D$ and adds $(i{-}k{-}1,j{+}1)\in D'$. We see that if the crossing $(i,j)\in D$ is of type $(a,b)$, then $(i{-}k{-}1,j{+}1)\in D'$ is of type $(a,b)$ as well, while all other crossings have the same type as before.
\begin{theorem}[\cite{bergeron1993RC}]
Any RC-graph $D\in\RC(w)$ can be obtained from $B_w$ by a sequence of ladder moves.
\end{theorem}
We adopt the following definition from \cite{weigandt2018schubert}.
\begin{definition}\label{def:label}
For $D\in\RC(w)$, let its \textit{label} $\alpha(D)$ be a vector such that $$\alpha(D)_k:=\{(i,j)\in D\ |\ i+j-1=k\}.$$
\end{definition}
The label $\alpha(D)$ records the number of crossings on each diagonal. We say that a crossing $(i,j)\in D$ is on diagonal $i+j-1$ and denote it by $\diag((i,j))=i+j-1$.

A ladder move is called \textit{simple} if its order is 0, and a ladder move is \textit{non-simple} if its order is positive. It is then clear that simple ladder moves do not change the label of RC-graphs, while non-simple ladder moves increase the label lexicographically. 

\section{The main theorem}\label{sec:main}
\begin{theorem}\label{thm:main}
For $w\in S_\infty$,
$\S_w(1)\geq 1+p_{132}(w)+p_{1432}(w).$
\end{theorem}

Theorem~\ref{thm:main} strengthens the following theorem by A. Weigandt.
\begin{theorem}\cite{weigandt2018schubert}\label{thm:weigandt}
For $w\in S_{\infty}$, $\S_w(1)\geq 1+p_{132}(w).$
\end{theorem}
Let's quickly review the proof idea of Theorem~\ref{thm:weigandt}: Weigandt \cite{weigandt2018schubert} showed that the top RC-graph $T_w$ and the bottom RC-graph $B_w$ are connected by a sequence of simple ladder moves and any such a chain of RC-graphs contains exactly $1+p_{132}(w)$ RC-graphs of $w$. Since simple ladder moves don't change the label of RC-graphs, to prove Theorem~\ref{thm:main}, it suffices to construct $p_{1432}(w)$ RC-graphs with a different label than $\alpha(B_w)$.

\begin{example}\label{ex:1432}
Figure~\ref{fig:1432} shows the Rothe diagram and all RC-graphs of the permutation $w=1432$. We see that $\#\RC(w)=5$, while $1+p_{132}(w)=4$, accounting for the first 4 RC-graphs (from left to right) listed in Figure~\ref{fig:1432}. The left most RC-graph is $B_w$ and the fourth is $T_w$. These four RC-graphs can be obtained from the previous one by a simple ladder move, while the fifth RC-graph is obtained from $B_w$ by a ladder move of order 1.
\begin{figure}[h!]
\centering
\begin{tikzpicture}[scale=0.200000000000000]
\draw(2,-2)--(10,-2)--(10,-10)--(2,-10)--(2,-2);
\draw[gray!20!white,ultra thin,step=2] (2,-2) grid (10,-10);
\node at (3,-3) {\footnotesize $\bullet$};
\draw(3,-10)--(3,-3)--(10,-3);
\node at (9,-5) {\footnotesize $\bullet$};
\draw(9,-10)--(9,-5)--(10,-5);
\node at (7,-7) {\footnotesize $\bullet$};
\draw(7,-10)--(7,-7)--(10,-7);
\node at (5,-9) {\footnotesize $\bullet$};
\draw(5,-10)--(5,-9)--(10,-9);
\fill[gray!50!white, draw=black] (4,-4) rectangle (6,-6);
\fill[gray!50!white, draw=black] (6,-4) rectangle (8,-6);
\fill[gray!50!white, draw=black] (4,-6) rectangle (6,-8);
\end{tikzpicture}
\qquad
\begin{tikzpicture}[scale=0.200000000000000]
\node at (3,-3) {\normalsize $\cdot$};
\node at (5,-3) {\normalsize $\cdot$};
\node at (7,-3) {\normalsize $\cdot$};
\node at (3,-5) {\normalsize $+$};
\node at (5,-5) {\normalsize $+$};
\node at (7,-5) {\normalsize $\cdot$};
\node at (3,-7) {\normalsize $+$};
\node at (5,-7) {\normalsize $\cdot$};
\node at (7,-7) {\normalsize $\cdot$};
\draw[gray!20!white,ultra thin,step=2] (2,-2) grid (8,-8);
\draw(2,-2)--(8,-2)--(8,-8)--(2,-8)--(2,-2);
\end{tikzpicture}
\quad
\begin{tikzpicture}[scale=0.200000000000000]
\node at (3,-3) {\normalsize $\cdot$};
\node at (5,-3) {\normalsize $\cdot$};
\node at (7,-3) {\normalsize $+$};
\node at (3,-5) {\normalsize $+$};
\node at (5,-5) {\normalsize $\cdot$};
\node at (7,-5) {\normalsize $\cdot$};
\node at (3,-7) {\normalsize $+$};
\node at (5,-7) {\normalsize $\cdot$};
\node at (7,-7) {\normalsize $\cdot$};
\draw[gray!20!white,ultra thin,step=2] (2,-2) grid (8,-8);
\draw(2,-2)--(8,-2)--(8,-8)--(2,-8)--(2,-2);
\end{tikzpicture}
\quad
\begin{tikzpicture}[scale=0.200000000000000]
\node at (3,-3) {\normalsize $\cdot$};
\node at (5,-3) {\normalsize $+$};
\node at (7,-3) {\normalsize $+$};
\node at (3,-5) {\normalsize $\cdot$};
\node at (5,-5) {\normalsize $\cdot$};
\node at (7,-5) {\normalsize $\cdot$};
\node at (3,-7) {\normalsize $+$};
\node at (5,-7) {\normalsize $\cdot$};
\node at (7,-7) {\normalsize $\cdot$};
\draw[gray!20!white,ultra thin,step=2] (2,-2) grid (8,-8);
\draw(2,-2)--(8,-2)--(8,-8)--(2,-8)--(2,-2);
\end{tikzpicture}
\quad
\begin{tikzpicture}[scale=0.200000000000000]
\node at (3,-3) {\normalsize $\cdot$};
\node at (5,-3) {\normalsize $+$};
\node at (7,-3) {\normalsize $+$};
\node at (3,-5) {\normalsize $\cdot$};
\node at (5,-5) {\normalsize $+$};
\node at (7,-5) {\normalsize $\cdot$};
\node at (3,-7) {\normalsize $\cdot$};
\node at (5,-7) {\normalsize $\cdot$};
\node at (7,-7) {\normalsize $\cdot$};
\draw[gray!20!white,ultra thin,step=2] (2,-2) grid (8,-8);
\draw(2,-2)--(8,-2)--(8,-8)--(2,-8)--(2,-2);
\end{tikzpicture}
\quad
\begin{tikzpicture}[scale=0.200000000000000]
\node at (3,-3) {\normalsize $\cdot$};
\node at (5,-3) {\normalsize $+$};
\node at (7,-3) {\normalsize $\cdot$};
\node at (3,-5) {\normalsize $+$};
\node at (5,-5) {\normalsize $+$};
\node at (7,-5) {\normalsize $\cdot$};
\node at (3,-7) {\normalsize $\cdot$};
\node at (5,-7) {\normalsize $\cdot$};
\node at (7,-7) {\normalsize $\cdot$};
\draw[gray!20!white,ultra thin,step=2] (2,-2) grid (8,-8);
\draw(2,-2)--(8,-2)--(8,-8)--(2,-8)--(2,-2);
\end{tikzpicture}
\caption{The Rothe diagram and all RC-graphs of $w=1432$.}
\label{fig:1432}
\end{figure}
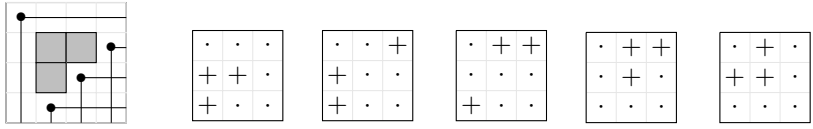
\end{example}

\begin{remark}
It is not true that ladder moves of order 1 (together with simple ladder moves) can produce $p_{1432}(w)$ many RC-graphs. To prove Theorem~\ref{thm:main}, we have to utilize ladder moves of high orders. For example, if $w=14532$, then $p_{1432}(w)=2$. But starting from $B_w$ and applying any sequence of ladder moves of order 0 or 1, we can only obtain one RC-graph (the right of Figure~\ref{fig:14532}) that has a different label than $B_w$.
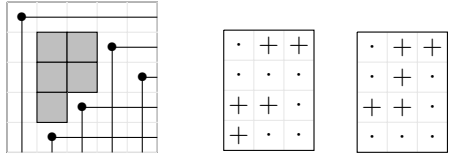
\begin{figure}[h!]
\centering
\begin{tikzpicture}[scale=0.200000000000000]
\draw(2,-2)--(12,-2)--(12,-12)--(2,-12)--(2,-2);
\draw[gray!20!white,ultra thin,step=2] (2,-2) grid (12,-12);
\node at (3,-3) {\footnotesize $\bullet$};
\draw(3,-12)--(3,-3)--(12,-3);
\node at (9,-5) {\footnotesize $\bullet$};
\draw(9,-12)--(9,-5)--(12,-5);
\node at (11,-7) {\footnotesize $\bullet$};
\draw(11,-12)--(11,-7)--(12,-7);
\node at (7,-9) {\footnotesize $\bullet$};
\draw(7,-12)--(7,-9)--(12,-9);
\node at (5,-11) {\footnotesize $\bullet$};
\draw(5,-12)--(5,-11)--(12,-11);
\fill[gray!50!white, draw=black] (4,-4) rectangle (6,-6);
\fill[gray!50!white, draw=black] (6,-4) rectangle (8,-6);
\fill[gray!50!white, draw=black] (4,-6) rectangle (6,-8);
\fill[gray!50!white, draw=black] (6,-6) rectangle (8,-8);
\fill[gray!50!white, draw=black] (4,-8) rectangle (6,-10);
\end{tikzpicture}
\qquad
\begin{tikzpicture}[scale=0.200000000000000]
\node at (3,-3) {\normalsize $\cdot$};
\node at (5,-3) {\normalsize $+$};
\node at (7,-3) {\normalsize $+$};
\node at (3,-5) {\normalsize $\cdot$};
\node at (5,-5) {\normalsize $\cdot$};
\node at (7,-5) {\normalsize $\cdot$};
\node at (3,-7) {\normalsize $+$};
\node at (5,-7) {\normalsize $+$};
\node at (7,-7) {\normalsize $\cdot$};
\node at (3,-9) {\normalsize $+$};
\node at (5,-9) {\normalsize $\cdot$};
\node at (7,-9) {\normalsize $\cdot$};
\draw[gray!20!white,ultra thin,step=2] (2,-2) grid (8,-10);
\draw(2,-2)--(8,-2)--(8,-10)--(2,-10)--(2,-2);
\end{tikzpicture}
\quad
\begin{tikzpicture}[scale=0.200000000000000]
\node at (3,-3) {\normalsize $\cdot$};
\node at (5,-3) {\normalsize $+$};
\node at (7,-3) {\normalsize $+$};
\node at (3,-5) {\normalsize $\cdot$};
\node at (5,-5) {\normalsize $+$};
\node at (7,-5) {\normalsize $\cdot$};
\node at (3,-7) {\normalsize $+$};
\node at (5,-7) {\normalsize $+$};
\node at (7,-7) {\normalsize $\cdot$};
\node at (3,-9) {\normalsize $\cdot$};
\node at (5,-9) {\normalsize $\cdot$};
\node at (7,-9) {\normalsize $\cdot$};
\draw[gray!20!white,ultra thin,step=2] (2,-2) grid (8,-10);
\draw(2,-2)--(8,-2)--(8,-10)--(2,-10)--(2,-2);
\end{tikzpicture}
\caption{The Rothe diagram and the only ladder move of order 1 for $w=14532$.}
\label{fig:14532}
\end{figure}
\end{remark}

For each $(i,j)\in RD(w)$, let $A_{(i,j)}:=\{a{<}i\ |\ w(a){<}j\}$ and $C_{(i,j)}:=\{i{<}c{<}w^{-1}(j)\ |\ j{<}w(c){<}w(i)\}$. Pictorially in the Rothe diagram, for the box $(i,j)\in RD(w)$, $A_{(i,j)}$ is the set of dots (permutation entries) to the top left of $(i,j)$ and $C_{(i,j)}$ is the set of dots inside the rectangle to the bottom right of $(i,j)$ with vertices at $(i,w(i))$ and $(w^{-1}(j),j)$.

After left-justification, $(i,j)\in RD(w)$ becomes $(i,j-\# A_{(i,j)})\in B_w$. Denote this map by $\beta:RD(w)\rightarrow B_w$. If $(i,j_1)$ and $(i,j_2)$, two boxes in the same row of $RD(w)$ with $j_1<j_2$, have no squares in between, then they become adjacent after left-justification in $B_w$. Notice that $\diag(\beta(i,j))=i+j-\# A_{(i,j)}-1$. This expression is symmetric in $i$ and $j$. If $(i_1,j)$ and $(i_2,j)$, two boxes in the same column of $RD(w)$ with $i_1<i_2$, have no squares in between, then we similarly have $\diag(\beta(i_1,j))+1=\diag(\beta(i_2,j))$. This is because $\# A_{(i_2,j)}-\# A_{(i_1,j)}$ must be $i_2-i_1-1$ for there to be no squares left in between of $(i_1,j)$ and $(i_2,j)$ in $RD(w)$.
\begin{lemma}\label{lem:diagdiff}
Let $(i,j),(i',j')\in RD(w)$ such that $i<i'$ and $j<j'$, then $\diag(\beta(i',j'))-\diag(\beta(i,j))\geq2$.
\end{lemma}
\begin{proof}
For $(i,j)\in RD(w)$, $w(i)>j$ and $w^{-1}(j)>i$ so
\begin{align*}
\diag(\beta(i,j))=&i+j-\#A_{(i,j)}-1\\
=&\#\{(a,w(a))\ |\ a\leq i\}+\#\{(a,w(a))\ |\ w(a)\leq j\}\\
&-\#\{(a,w(a))\ |\ a<i,w(a)<j\}-1\\
=&\#\{(a,w(a))\ |\ a\leq i\text{ or }w(a)\leq j\}-1\\
=&n-1-\#\{(a,w(a))\ |\ a>i,\ w(a)>j\}.
\end{align*}
If $i'>i$ and $j'>j$, then $\{(a,w(a))\ |\ a>i,w(a)>j\}$ contains $\{(a,w(a))\ |\ a>i',w(a)>j'\}$ disjoint union with $\{(i,w(i)),(w^{-1}(j),j)\}$ so $\diag(\beta(i',j'))-\diag(\beta(i,j))\geq2$.
\end{proof}
Intuitively, Lemma~\ref{lem:diagdiff} is saying that if a box $(i',j')$ lies strictly to the bottom right of $(i,j)$ in $RD(w)$, then after left-justification to $B_w$, they won't interact with each other when applying simple ladder moves.

\begin{proof}[Proof of Theorem~\ref{thm:main}]
Consider the following calculation
\begin{align*}
p_{1432}(w)=&\sum_{b<d,w(b)>w(d)}\#\{a{<}b\ |\ w(a){<}w(d)\}\cdot\#\{b{<}c{<}d\ |\ w(d){<}w(c){<}w(b)\}\\
=&\sum_{(i,j)\in RD(w)}\#\{a{<}i\ |\ w(a){<}j\}\cdot\#\{i{<}c{<}w^{-1}(j)\ |\ j{<}w(c){<}w(i)\}.
\end{align*}
We first summarize our strategy. We will construct $\P_{(i,j)}\subset\RC(w)$, a set of RC-graphs of $w$ with different labels than that of $B_w$, whose cardinality is at least $\# A_{(i,j)}\cdot\# C_{(i,j)}$. We finish the proof by showing that such sets $\P_{(i,j)}$ don't intersect for distinct $(i,j)\in RD(w)$.

Now fix $(i,j)\in RD(w)$. Let $C_{(i,j)}$ be $\{c_1,\ldots,c_q\}$, ordered (from left to right) such that $j<w(c_1)<\cdots<w(c_q)<w(i)$. Notice that we don't have to do anything if $C_{(i,j)}=\emptyset$. The choice of $c_1$ makes sure that there are no $k$ such that $i\leq k\leq c_1$ and $j\leq w(k)\leq w(c_1)$ except $k=c_1$ or in other words, there are no dots (permutation entries) strictly inside the rectangle from $(i,j)$ to $(c_1,w(c_1))$ in the Rothe diagram $RD(w)$.

We will now construct a special RC-graph $D_{(i,j)}\in \RC(w)$ via simple ladder moves from $B_w$ such that non-simple ladder moves are ready to be applied to $D_{(i,j)}$. Let $I=\{i'\ |\ i\leq i'\leq c_1,\ (i',j)\in RD(w)\}$ and $J=\{j'\ |\ j\leq j'\leq w(c_1),\ (i,j')\in RD(w)\}$. Since there are no permutation entries in the rectangle in $RD(w)$ from $(i,j)$ to $(c_1,w(c_1))$, all squares inside this region are exactly $I\times J\setminus\{(c_1,w(c_1))\}$. See Figure~\ref{fig:AIJc1} for an example of these squares. Let $I=\{i=i_0<i_1<\cdots<i_r=c_1\}$, $r\geq1$ and $J=\{j=j_0<j_1<\cdots<j_m=w(c_1)\}$, $m\geq1$. For each $i_k\in I$, in increasing order of $k$ from $0$ to $r$, and for each $j'\geq j$ such that $(i_k,j')\in RD(w)$ (not just $j'\in J$) in decreasing order of $j'$, apply a simple ladder move $i_k-i-k$ times in $B_w$ to $\beta(i_k,j')$. Let this RC-graph be $D_{(i,j)}^{(0)}$. See Figure~\ref{fig:AIJc1} for an example. Intuitively, from $B_w$, we apply simple ladder moves to row $i_k$ starting from the crossing $\beta(i_k,j)$ enough times such that the squares $I\times J\setminus\{(c_1,w(c_1))\}\subset RD(w)$ form a consecutive rectangle without a corner in $D_{(i,j)}^{(0)}$. To see that such process is possible, it suffices to show that $(s,t)\in RD(w)$ won't interfere these simple moves, for some $i\leq s\leq c_1$. If $s\notin I$, by definition of $I$, $w(s)<j$ so $t<j$. Then for $(i',j')\in I\times J\setminus\{c_1,w(c_1)\}$, if $i'<s$, clearly $(s,t)$ won't be in the way since $\beta(i',j')$ are moving up by simple ladder moves but $(s,t)$ is lower (at a larger row number) to start with, and if $i'>s$, with $j'\geq j>t$, by Lemma~\ref{lem:diagdiff}, we are done as well. The case $s\in I$ is argued in the exact same way. Recall that if $(i',j')\in RD(w)$ and $(i'',j'')\in RD(w)$ are in the same row (or column) and there are no other boxes in between, then in $B_w$, the diagonals where they are on differ by exactly 1. As a result, from $B_w$ to $D_{(i,j)}^{(0)}$, we have moved $\beta(I\times J\setminus\{(c_1,w(c_1))\})$ to $\{i,i+1,\ldots,i+r\}\times\{l,l+1\ldots,l+m\}\setminus\{(i+r,l+m)\}$ where $\beta(i,j)=(i,l)$.
\begin{figure}[h!]
\centering
\begin{tikzpicture}[scale=0.183000000000000]
\draw(2,-2)--(22,-2)--(22,-22)--(2,-22)--(2,-2);
\draw[gray!20!white,ultra thin,step=2] (2,-2) grid (22,-22);
\node at (7,-3) {\footnotesize $\bullet$};
\draw(7,-22)--(7,-3)--(22,-3);
\node at (19,-5) {\footnotesize $\bullet$};
\draw(19,-22)--(19,-5)--(22,-5);
\node at (5,-7) {\footnotesize $\bullet$};
\draw(5,-22)--(5,-7)--(22,-7);
\node at (21,-9) {\footnotesize $\bullet$};
\draw(21,-22)--(21,-9)--(22,-9);
\node at (3,-11) {\footnotesize $\bullet$};
\draw(3,-22)--(3,-11)--(22,-11);
\node at (17,-13) {\footnotesize $\bullet$};
\draw(17,-22)--(17,-13)--(22,-13);
\node at (11,-15) {\footnotesize $\bullet$};
\draw(11,-22)--(11,-15)--(22,-15);
\node at (15,-17) {\footnotesize $\bullet$};
\draw(15,-22)--(15,-17)--(22,-17);
\node at (9,-19) {\footnotesize $\bullet$};
\draw(9,-22)--(9,-19)--(22,-19);
\node at (13,-21) {\footnotesize $\bullet$};
\draw(13,-22)--(13,-21)--(22,-21);
\fill[gray!50!white, draw=black] (2,-2) rectangle (4,-4);
\fill[gray!50!white, draw=black] (4,-2) rectangle (6,-4);
\fill[gray!50!white, draw=black] (2,-4) rectangle (4,-6);
\fill[gray!50!white, draw=black] (4,-4) rectangle (6,-6);
\fill[gray!50!white, draw=black] (8,-4) rectangle (10,-6);
\fill[gray!50!white, draw=black] (10,-4) rectangle (12,-6);
\fill[gray!50!white, draw=black] (12,-4) rectangle (14,-6);
\fill[gray!50!white, draw=black] (14,-4) rectangle (16,-6);
\fill[gray!50!white, draw=black] (16,-4) rectangle (18,-6);
\fill[gray!50!white, draw=black] (2,-6) rectangle (4,-8);
\fill[gray!50!white, draw=black] (2,-8) rectangle (4,-10);
\fill[gray!50!white, draw=black] (8,-8) rectangle (10,-10);
\fill[gray!50!white, draw=black] (10,-8) rectangle (12,-10);
\fill[gray!50!white, draw=black] (12,-8) rectangle (14,-10);
\fill[gray!50!white, draw=black] (14,-8) rectangle (16,-10);
\fill[gray!50!white, draw=black] (16,-8) rectangle (18,-10);
\fill[gray!50!white, draw=black] (8,-12) rectangle (10,-14);
\fill[gray!50!white, draw=black] (10,-12) rectangle (12,-14);
\fill[gray!50!white, draw=black] (12,-12) rectangle (14,-14);
\fill[gray!50!white, draw=black] (14,-12) rectangle (16,-14);
\fill[gray!50!white, draw=black] (8,-14) rectangle (10,-16);
\fill[gray!50!white, draw=black] (8,-16) rectangle (10,-18);
\fill[gray!50!white, draw=black] (12,-16) rectangle (14,-18);
\draw(8,-8)--(10,-10);
\draw(10,-8)--(8,-10);
\draw(10,-8)--(12,-10);
\draw(12,-8)--(10,-10);
\draw(8,-12)--(10,-14);
\draw(10,-12)--(8,-14);
\draw(10,-12)--(12,-14);
\draw(12,-12)--(10,-14);
\draw(8,-14)--(10,-16);
\draw(10,-14)--(8,-16);
\node at (1,-9) {$i$};
\node at (9,-1) {$j$};
\node at (1,-15) {$c_1$};
\end{tikzpicture}
\quad
\begin{tikzpicture}[scale=0.183000000000000]
\node at (3,-3) {\normalsize $+$};
\node at (5,-3) {\normalsize $+$};
\node at (7,-3) {\normalsize $\cdot$};
\node at (9,-3) {\normalsize $\cdot$};
\node at (11,-3) {\normalsize $\cdot$};
\node at (13,-3) {\normalsize $\cdot$};
\node at (15,-3) {\normalsize $\cdot$};
\node at (17,-3) {\normalsize $\cdot$};
\node at (19,-3) {\normalsize $\cdot$};
\node at (21,-3) {\normalsize $\cdot$};
\node at (3,-5) {\normalsize $+$};
\node at (5,-5) {\normalsize $+$};
\node at (7,-5) {\normalsize $+$};
\node at (9,-5) {\normalsize $+$};
\node at (11,-5) {\normalsize $+$};
\node at (13,-5) {\normalsize $+$};
\node at (15,-5) {\normalsize $+$};
\node at (17,-5) {\normalsize $\cdot$};
\node at (19,-5) {\normalsize $\cdot$};
\node at (21,-5) {\normalsize $\cdot$};
\node at (3,-7) {\normalsize $+$};
\node at (5,-7) {\normalsize $\cdot$};
\node at (7,-7) {\normalsize $\cdot$};
\node at (9,-7) {\normalsize $\cdot$};
\node at (11,-7) {\normalsize $\cdot$};
\node at (13,-7) {\normalsize $\cdot$};
\node at (15,-7) {\normalsize $\cdot$};
\node at (17,-7) {\normalsize $\cdot$};
\node at (19,-7) {\normalsize $\cdot$};
\node at (21,-7) {\normalsize $\cdot$};
\node at (3,-9) {\normalsize $+$};
\node at (5,-9) {\normalsize $+$};
\node at (7,-9) {\normalsize $+$};
\node at (9,-9) {\normalsize $+$};
\node at (11,-9) {\normalsize $+$};
\node at (13,-9) {\normalsize $+$};
\node at (15,-9) {\normalsize $\cdot$};
\node at (17,-9) {\normalsize $\cdot$};
\node at (19,-9) {\normalsize $\cdot$};
\node at (21,-9) {\normalsize $\cdot$};
\node at (3,-11) {\normalsize $\cdot$};
\node at (5,-11) {\normalsize $\cdot$};
\node at (7,-11) {\normalsize $\cdot$};
\node at (9,-11) {\normalsize $\cdot$};
\node at (11,-11) {\normalsize $\cdot$};
\node at (13,-11) {\normalsize $\cdot$};
\node at (15,-11) {\normalsize $\cdot$};
\node at (17,-11) {\normalsize $\cdot$};
\node at (19,-11) {\normalsize $\cdot$};
\node at (21,-11) {\normalsize $\cdot$};
\node at (3,-13) {\normalsize $+$};
\node at (5,-13) {\normalsize $+$};
\node at (7,-13) {\normalsize $+$};
\node at (9,-13) {\normalsize $+$};
\node at (11,-13) {\normalsize $\cdot$};
\node at (13,-13) {\normalsize $\cdot$};
\node at (15,-13) {\normalsize $\cdot$};
\node at (17,-13) {\normalsize $\cdot$};
\node at (19,-13) {\normalsize $\cdot$};
\node at (21,-13) {\normalsize $\cdot$};
\node at (3,-15) {\normalsize $+$};
\node at (5,-15) {\normalsize $\cdot$};
\node at (7,-15) {\normalsize $\cdot$};
\node at (9,-15) {\normalsize $\cdot$};
\node at (11,-15) {\normalsize $\cdot$};
\node at (13,-15) {\normalsize $\cdot$};
\node at (15,-15) {\normalsize $\cdot$};
\node at (17,-15) {\normalsize $\cdot$};
\node at (19,-15) {\normalsize $\cdot$};
\node at (21,-15) {\normalsize $\cdot$};
\node at (3,-17) {\normalsize $+$};
\node at (5,-17) {\normalsize $+$};
\node at (7,-17) {\normalsize $\cdot$};
\node at (9,-17) {\normalsize $\cdot$};
\node at (11,-17) {\normalsize $\cdot$};
\node at (13,-17) {\normalsize $\cdot$};
\node at (15,-17) {\normalsize $\cdot$};
\node at (17,-17) {\normalsize $\cdot$};
\node at (19,-17) {\normalsize $\cdot$};
\node at (21,-17) {\normalsize $\cdot$};
\node at (3,-19) {\normalsize $\cdot$};
\node at (5,-19) {\normalsize $\cdot$};
\node at (7,-19) {\normalsize $\cdot$};
\node at (9,-19) {\normalsize $\cdot$};
\node at (11,-19) {\normalsize $\cdot$};
\node at (13,-19) {\normalsize $\cdot$};
\node at (15,-19) {\normalsize $\cdot$};
\node at (17,-19) {\normalsize $\cdot$};
\node at (19,-19) {\normalsize $\cdot$};
\node at (21,-19) {\normalsize $\cdot$};
\node at (3,-21) {\normalsize $\cdot$};
\node at (5,-21) {\normalsize $\cdot$};
\node at (7,-21) {\normalsize $\cdot$};
\node at (9,-21) {\normalsize $\cdot$};
\node at (11,-21) {\normalsize $\cdot$};
\node at (13,-21) {\normalsize $\cdot$};
\node at (15,-21) {\normalsize $\cdot$};
\node at (17,-21) {\normalsize $\cdot$};
\node at (19,-21) {\normalsize $\cdot$};
\node at (21,-21) {\normalsize $\cdot$};
\draw[gray!20!white,ultra thin,step=2] (2,-2) grid (22,-22);
\draw(2,-2)--(22,-2)--(22,-22)--(2,-22)--(2,-2);
\node at (5,-9) {\normalsize $\bigoplus$};
\node at (7,-9) {\normalsize $\bigoplus$};
\node at (3,-13) {\normalsize $\bigoplus$};
\node at (5,-13) {\normalsize $\bigoplus$};
\node at (3,-15) {\normalsize $\bigoplus$};
\end{tikzpicture}
\quad
\begin{tikzpicture}[scale=0.183000000000000]
\node at (3,-3) {\normalsize $+$};
\node at (5,-3) {\normalsize $+$};
\node at (7,-3) {\normalsize $\cdot$};
\node at (9,-3) {\normalsize $\cdot$};
\node at (11,-3) {\normalsize $\cdot$};
\node at (13,-3) {\normalsize $\cdot$};
\node at (15,-3) {\normalsize $\cdot$};
\node at (17,-3) {\normalsize $\cdot$};
\node at (19,-3) {\normalsize $\cdot$};
\node at (21,-3) {\normalsize $\cdot$};
\node at (3,-5) {\normalsize $+$};
\node at (5,-5) {\normalsize $+$};
\node at (7,-5) {\normalsize $+$};
\node at (9,-5) {\normalsize $+$};
\node at (11,-5) {\normalsize $+$};
\node at (13,-5) {\normalsize $+$};
\node at (15,-5) {\normalsize $+$};
\node at (17,-5) {\normalsize $\cdot$};
\node at (19,-5) {\normalsize $\cdot$};
\node at (21,-5) {\normalsize $\cdot$};
\node at (3,-7) {\normalsize $+$};
\node at (5,-7) {\normalsize $\cdot$};
\node at (7,-7) {\normalsize $\cdot$};
\node at (9,-7) {\normalsize $\cdot$};
\node at (11,-7) {\normalsize $\cdot$};
\node at (13,-7) {\normalsize $\cdot$};
\node at (15,-7) {\normalsize $\cdot$};
\node at (17,-7) {\normalsize $\cdot$};
\node at (19,-7) {\normalsize $\cdot$};
\node at (21,-7) {\normalsize $\cdot$};
\node at (3,-9) {\normalsize $+$};
\node at (5,-9) {\normalsize $+$};
\node at (7,-9) {\normalsize $+$};
\node at (9,-9) {\normalsize $+$};
\node at (11,-9) {\normalsize $+$};
\node at (13,-9) {\normalsize $+$};
\node at (15,-9) {\normalsize $\cdot$};
\node at (17,-9) {\normalsize $\cdot$};
\node at (19,-9) {\normalsize $\cdot$};
\node at (21,-9) {\normalsize $\cdot$};
\node at (3,-11) {\normalsize $\cdot$};
\node at (5,-11) {\normalsize $+$};
\node at (7,-11) {\normalsize $+$};
\node at (9,-11) {\normalsize $+$};
\node at (11,-11) {\normalsize $+$};
\node at (13,-11) {\normalsize $\cdot$};
\node at (15,-11) {\normalsize $\cdot$};
\node at (17,-11) {\normalsize $\cdot$};
\node at (19,-11) {\normalsize $\cdot$};
\node at (21,-11) {\normalsize $\cdot$};
\node at (3,-13) {\normalsize $\cdot$};
\node at (5,-13) {\normalsize $+$};
\node at (7,-13) {\normalsize $\cdot$};
\node at (9,-13) {\normalsize $\cdot$};
\node at (11,-13) {\normalsize $\cdot$};
\node at (13,-13) {\normalsize $\cdot$};
\node at (15,-13) {\normalsize $\cdot$};
\node at (17,-13) {\normalsize $\cdot$};
\node at (19,-13) {\normalsize $\cdot$};
\node at (21,-13) {\normalsize $\cdot$};
\node at (3,-15) {\normalsize $\cdot$};
\node at (5,-15) {\normalsize $\cdot$};
\node at (7,-15) {\normalsize $\cdot$};
\node at (9,-15) {\normalsize $\cdot$};
\node at (11,-15) {\normalsize $\cdot$};
\node at (13,-15) {\normalsize $\cdot$};
\node at (15,-15) {\normalsize $\cdot$};
\node at (17,-15) {\normalsize $\cdot$};
\node at (19,-15) {\normalsize $\cdot$};
\node at (21,-15) {\normalsize $\cdot$};
\node at (3,-17) {\normalsize $+$};
\node at (5,-17) {\normalsize $+$};
\node at (7,-17) {\normalsize $\cdot$};
\node at (9,-17) {\normalsize $\cdot$};
\node at (11,-17) {\normalsize $\cdot$};
\node at (13,-17) {\normalsize $\cdot$};
\node at (15,-17) {\normalsize $\cdot$};
\node at (17,-17) {\normalsize $\cdot$};
\node at (19,-17) {\normalsize $\cdot$};
\node at (21,-17) {\normalsize $\cdot$};
\node at (3,-19) {\normalsize $\cdot$};
\node at (5,-19) {\normalsize $\cdot$};
\node at (7,-19) {\normalsize $\cdot$};
\node at (9,-19) {\normalsize $\cdot$};
\node at (11,-19) {\normalsize $\cdot$};
\node at (13,-19) {\normalsize $\cdot$};
\node at (15,-19) {\normalsize $\cdot$};
\node at (17,-19) {\normalsize $\cdot$};
\node at (19,-19) {\normalsize $\cdot$};
\node at (21,-19) {\normalsize $\cdot$};
\node at (3,-21) {\normalsize $\cdot$};
\node at (5,-21) {\normalsize $\cdot$};
\node at (7,-21) {\normalsize $\cdot$};
\node at (9,-21) {\normalsize $\cdot$};
\node at (11,-21) {\normalsize $\cdot$};
\node at (13,-21) {\normalsize $\cdot$};
\node at (15,-21) {\normalsize $\cdot$};
\node at (17,-21) {\normalsize $\cdot$};
\node at (19,-21) {\normalsize $\cdot$};
\node at (21,-21) {\normalsize $\cdot$};
\draw[gray!20!white,ultra thin,step=2] (2,-2) grid (22,-22);
\draw(2,-2)--(22,-2)--(22,-22)--(2,-22)--(2,-2);
\node at (5,-9) {\normalsize $\bigoplus$};
\node at (7,-9) {\normalsize $\bigoplus$};
\node at (5,-11) {\normalsize $\bigoplus$};
\node at (7,-11) {\normalsize $\bigoplus$};
\node at (5,-13) {\normalsize $\bigoplus$};
\end{tikzpicture}
\caption{The Rothe diagram $RD(w)$, the bottom RC-graph $B_w$ and $D_{(i,j)}^{(0)}$ for $w=3,9,2,10,1,8,5,7,4,6$, where $(i,j)=(4,4)$, $\#A_{(i,j)}=2$, $\#C_{(i,j)}=3$, $c_1=7$, $I=\{4,6,7\}$, $J=\{4,5\}$ and squares $I\times J\setminus\{(c_1,w(c_1))\}$ are marked.}
\label{fig:AIJc1}
\end{figure}
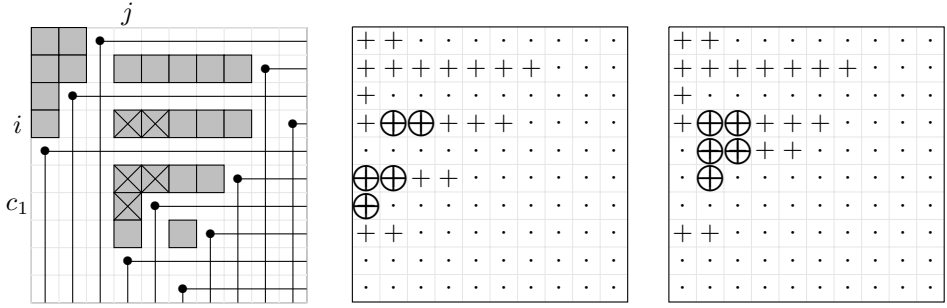

From $D_{(i,j)}^{0}$, for each $i'<i$ in increasing order and for each $j'\geq j$ in decreasing order, we apply simple ladder moves $\#\{a<i'\ |\ w(a)<j\}$ times to $\beta(i',j')$ if $(i',j')\in RD(w)$. Here we abuse the notation of $\beta(-,-)$ to also mean squares in $D_{(i,j)}^{(0)}$ since previously we moved squares after row $i$ and now we move squares before row $i$. Call this new RC-graph $D_{(i,j)}$. The purpose of this step is to make sure that the key squares $\{i,\ldots,i+r\}\times\{l,\ldots,l+m\}\setminus\{(i+r,l+m)\}$ in $D_{(i,j)}$ have enough room above row $i$ to move around. See Figure~\ref{fig:Dij} as a continuing example from Figure~\ref{fig:AIJc1}. More formally, for $(i',j')\in RD(w)$ with $i'<i$, if $j'<j$, then this square has never moved from $B_w$ to $D_{(i,j)}$ and by Lemma~\ref{lem:diagdiff}, $\diag(\beta(i',j'))-\diag(\beta(i,j))\leq -2$ and if $j'\geq j$, then its corresponding square in $D_{(i,j)}$ has a row number at most $i'-\#\{a<i'\ |\ w(a)<j\}\leq i-1-\#A_{(i,j)}$. As a result, $(i',j')\notin D_{(i,j)}$ for all $i-\#A_{(i,j)}\leq i'\leq i-1$ and $j'\geq (i+j-1)-i'$. These important empty squares are also marked in the running example in Figure~\ref{fig:Dij}.

\begin{figure}[h!]
\centering
\begin{tikzpicture}[scale=0.183000000000000]
\node at (3,-3) {\normalsize $+$};
\node at (5,-3) {\normalsize $+$};
\node at (7,-3) {\normalsize $\cdot$};
\node at (9,-3) {\normalsize $\cdot$};
\node at (11,-3) {\normalsize $\cdot$};
\node at (13,-3) {\normalsize $\cdot$};
\node at (15,-3) {\normalsize $\cdot$};
\node at (17,-3) {\normalsize $\cdot$};
\node at (19,-3) {\normalsize $\cdot$};
\node at (21,-3) {\normalsize $\cdot$};
\node at (3,-5) {\normalsize $+$};
\node at (5,-5) {\normalsize $+$};
\node at (7,-5) {\normalsize $+$};
\node at (9,-5) {\normalsize $+$};
\node at (11,-5) {\normalsize $+$};
\node at (13,-5) {\normalsize $+$};
\node at (15,-5) {\normalsize $+$};
\node at (17,-5) {\normalsize $\cdot$};
\node at (19,-5) {\normalsize $\cdot$};
\node at (21,-5) {\normalsize $\cdot$};
\node at (3,-7) {\normalsize $+$};
\node at (5,-7) {\normalsize $\cdot$};
\node at (7,-7) {\normalsize $\cdot$};
\node at (9,-7) {\normalsize $\cdot$};
\node at (11,-7) {\normalsize $\cdot$};
\node at (13,-7) {\normalsize $\cdot$};
\node at (15,-7) {\normalsize $\cdot$};
\node at (17,-7) {\normalsize $\cdot$};
\node at (19,-7) {\normalsize $\cdot$};
\node at (21,-7) {\normalsize $\cdot$};
\node at (3,-9) {\normalsize $+$};
\node at (5,-9) {\normalsize $+$};
\node at (7,-9) {\normalsize $+$};
\node at (9,-9) {\normalsize $+$};
\node at (11,-9) {\normalsize $+$};
\node at (13,-9) {\normalsize $+$};
\node at (15,-9) {\normalsize $\cdot$};
\node at (17,-9) {\normalsize $\cdot$};
\node at (19,-9) {\normalsize $\cdot$};
\node at (21,-9) {\normalsize $\cdot$};
\node at (3,-11) {\normalsize $\cdot$};
\node at (5,-11) {\normalsize $+$};
\node at (7,-11) {\normalsize $+$};
\node at (9,-11) {\normalsize $+$};
\node at (11,-11) {\normalsize $+$};
\node at (13,-11) {\normalsize $\cdot$};
\node at (15,-11) {\normalsize $\cdot$};
\node at (17,-11) {\normalsize $\cdot$};
\node at (19,-11) {\normalsize $\cdot$};
\node at (21,-11) {\normalsize $\cdot$};
\node at (3,-13) {\normalsize $\cdot$};
\node at (5,-13) {\normalsize $+$};
\node at (7,-13) {\normalsize $\cdot$};
\node at (9,-13) {\normalsize $\cdot$};
\node at (11,-13) {\normalsize $\cdot$};
\node at (13,-13) {\normalsize $\cdot$};
\node at (15,-13) {\normalsize $\cdot$};
\node at (17,-13) {\normalsize $\cdot$};
\node at (19,-13) {\normalsize $\cdot$};
\node at (21,-13) {\normalsize $\cdot$};
\node at (3,-15) {\normalsize $\cdot$};
\node at (5,-15) {\normalsize $\cdot$};
\node at (7,-15) {\normalsize $\cdot$};
\node at (9,-15) {\normalsize $\cdot$};
\node at (11,-15) {\normalsize $\cdot$};
\node at (13,-15) {\normalsize $\cdot$};
\node at (15,-15) {\normalsize $\cdot$};
\node at (17,-15) {\normalsize $\cdot$};
\node at (19,-15) {\normalsize $\cdot$};
\node at (21,-15) {\normalsize $\cdot$};
\node at (3,-17) {\normalsize $+$};
\node at (5,-17) {\normalsize $+$};
\node at (7,-17) {\normalsize $\cdot$};
\node at (9,-17) {\normalsize $\cdot$};
\node at (11,-17) {\normalsize $\cdot$};
\node at (13,-17) {\normalsize $\cdot$};
\node at (15,-17) {\normalsize $\cdot$};
\node at (17,-17) {\normalsize $\cdot$};
\node at (19,-17) {\normalsize $\cdot$};
\node at (21,-17) {\normalsize $\cdot$};
\node at (3,-19) {\normalsize $\cdot$};
\node at (5,-19) {\normalsize $\cdot$};
\node at (7,-19) {\normalsize $\cdot$};
\node at (9,-19) {\normalsize $\cdot$};
\node at (11,-19) {\normalsize $\cdot$};
\node at (13,-19) {\normalsize $\cdot$};
\node at (15,-19) {\normalsize $\cdot$};
\node at (17,-19) {\normalsize $\cdot$};
\node at (19,-19) {\normalsize $\cdot$};
\node at (21,-19) {\normalsize $\cdot$};
\node at (3,-21) {\normalsize $\cdot$};
\node at (5,-21) {\normalsize $\cdot$};
\node at (7,-21) {\normalsize $\cdot$};
\node at (9,-21) {\normalsize $\cdot$};
\node at (11,-21) {\normalsize $\cdot$};
\node at (13,-21) {\normalsize $\cdot$};
\node at (15,-21) {\normalsize $\cdot$};
\node at (17,-21) {\normalsize $\cdot$};
\node at (19,-21) {\normalsize $\cdot$};
\node at (21,-21) {\normalsize $\cdot$};
\draw[gray!20!white,ultra thin,step=2] (2,-2) grid (22,-22);
\draw(2,-2)--(22,-2)--(22,-22)--(2,-22)--(2,-2);
\node at (5,-9) {\normalsize $\bigoplus$};
\node at (7,-9) {\normalsize $\bigoplus$};
\node at (5,-11) {\normalsize $\bigoplus$};
\node at (7,-11) {\normalsize $\bigoplus$};
\node at (5,-13) {\normalsize $\bigoplus$};
\end{tikzpicture}
\qquad
\begin{tikzpicture}[scale=0.183000000000000]
\fill[gray!20!white] (4,-6) rectangle (6,-8);
\fill[gray!20!white] (6,-6) rectangle (8,-8);
\fill[gray!20!white] (8,-6) rectangle (10,-8);
\fill[gray!20!white] (10,-6) rectangle (12,-8);
\fill[gray!20!white] (12,-6) rectangle (14,-8);
\fill[gray!20!white] (14,-6) rectangle (16,-8);
\fill[gray!20!white] (16,-6) rectangle (18,-8);
\fill[gray!20!white] (18,-6) rectangle (20,-8);
\fill[gray!20!white] (20,-6) rectangle (22,-8);
\fill[gray!20!white] (6,-4) rectangle (8,-6);
\fill[gray!20!white] (8,-4) rectangle (10,-6);
\fill[gray!20!white] (10,-4) rectangle (12,-6);
\fill[gray!20!white] (12,-4) rectangle (14,-6);
\fill[gray!20!white] (14,-4) rectangle (16,-6);
\fill[gray!20!white] (16,-4) rectangle (18,-6);
\fill[gray!20!white] (18,-4) rectangle (20,-6);
\fill[gray!20!white] (20,-4) rectangle (22,-6);
\node at (3,-3) {\normalsize $+$};
\node at (5,-3) {\normalsize $+$};
\node at (7,-3) {\normalsize $\cdot$};
\node at (9,-3) {\normalsize $+$};
\node at (11,-3) {\normalsize $+$};
\node at (13,-3) {\normalsize $+$};
\node at (15,-3) {\normalsize $+$};
\node at (17,-3) {\normalsize $+$};
\node at (19,-3) {\normalsize $\cdot$};
\node at (21,-3) {\normalsize $\cdot$};
\node at (3,-5) {\normalsize $+$};
\node at (5,-5) {\normalsize $+$};
\node at (7,-5) {\normalsize $\cdot$};
\node at (9,-5) {\normalsize $\cdot$};
\node at (11,-5) {\normalsize $\cdot$};
\node at (13,-5) {\normalsize $\cdot$};
\node at (15,-5) {\normalsize $\cdot$};
\node at (17,-5) {\normalsize $\cdot$};
\node at (19,-5) {\normalsize $\cdot$};
\node at (21,-5) {\normalsize $\cdot$};
\node at (3,-7) {\normalsize $+$};
\node at (5,-7) {\normalsize $\cdot$};
\node at (7,-7) {\normalsize $\cdot$};
\node at (9,-7) {\normalsize $\cdot$};
\node at (11,-7) {\normalsize $\cdot$};
\node at (13,-7) {\normalsize $\cdot$};
\node at (15,-7) {\normalsize $\cdot$};
\node at (17,-7) {\normalsize $\cdot$};
\node at (19,-7) {\normalsize $\cdot$};
\node at (21,-7) {\normalsize $\cdot$};
\node at (3,-9) {\normalsize $+$};
\node at (5,-9) {\normalsize $+$};
\node at (7,-9) {\normalsize $+$};
\node at (9,-9) {\normalsize $+$};
\node at (11,-9) {\normalsize $+$};
\node at (13,-9) {\normalsize $+$};
\node at (15,-9) {\normalsize $\cdot$};
\node at (17,-9) {\normalsize $\cdot$};
\node at (19,-9) {\normalsize $\cdot$};
\node at (21,-9) {\normalsize $\cdot$};
\node at (3,-11) {\normalsize $\cdot$};
\node at (5,-11) {\normalsize $+$};
\node at (7,-11) {\normalsize $+$};
\node at (9,-11) {\normalsize $+$};
\node at (11,-11) {\normalsize $+$};
\node at (13,-11) {\normalsize $\cdot$};
\node at (15,-11) {\normalsize $\cdot$};
\node at (17,-11) {\normalsize $\cdot$};
\node at (19,-11) {\normalsize $\cdot$};
\node at (21,-11) {\normalsize $\cdot$};
\node at (3,-13) {\normalsize $\cdot$};
\node at (5,-13) {\normalsize $+$};
\node at (7,-13) {\normalsize $\cdot$};
\node at (9,-13) {\normalsize $\cdot$};
\node at (11,-13) {\normalsize $\cdot$};
\node at (13,-13) {\normalsize $\cdot$};
\node at (15,-13) {\normalsize $\cdot$};
\node at (17,-13) {\normalsize $\cdot$};
\node at (19,-13) {\normalsize $\cdot$};
\node at (21,-13) {\normalsize $\cdot$};
\node at (3,-15) {\normalsize $\cdot$};
\node at (5,-15) {\normalsize $\cdot$};
\node at (7,-15) {\normalsize $\cdot$};
\node at (9,-15) {\normalsize $\cdot$};
\node at (11,-15) {\normalsize $\cdot$};
\node at (13,-15) {\normalsize $\cdot$};
\node at (15,-15) {\normalsize $\cdot$};
\node at (17,-15) {\normalsize $\cdot$};
\node at (19,-15) {\normalsize $\cdot$};
\node at (21,-15) {\normalsize $\cdot$};
\node at (3,-17) {\normalsize $+$};
\node at (5,-17) {\normalsize $+$};
\node at (7,-17) {\normalsize $\cdot$};
\node at (9,-17) {\normalsize $\cdot$};
\node at (11,-17) {\normalsize $\cdot$};
\node at (13,-17) {\normalsize $\cdot$};
\node at (15,-17) {\normalsize $\cdot$};
\node at (17,-17) {\normalsize $\cdot$};
\node at (19,-17) {\normalsize $\cdot$};
\node at (21,-17) {\normalsize $\cdot$};
\node at (3,-19) {\normalsize $\cdot$};
\node at (5,-19) {\normalsize $\cdot$};
\node at (7,-19) {\normalsize $\cdot$};
\node at (9,-19) {\normalsize $\cdot$};
\node at (11,-19) {\normalsize $\cdot$};
\node at (13,-19) {\normalsize $\cdot$};
\node at (15,-19) {\normalsize $\cdot$};
\node at (17,-19) {\normalsize $\cdot$};
\node at (19,-19) {\normalsize $\cdot$};
\node at (21,-19) {\normalsize $\cdot$};
\node at (3,-21) {\normalsize $\cdot$};
\node at (5,-21) {\normalsize $\cdot$};
\node at (7,-21) {\normalsize $\cdot$};
\node at (9,-21) {\normalsize $\cdot$};
\node at (11,-21) {\normalsize $\cdot$};
\node at (13,-21) {\normalsize $\cdot$};
\node at (15,-21) {\normalsize $\cdot$};
\node at (17,-21) {\normalsize $\cdot$};
\node at (19,-21) {\normalsize $\cdot$};
\node at (21,-21) {\normalsize $\cdot$};
\draw[gray!20!white,ultra thin,step=2] (2,-2) grid (22,-22);
\draw(2,-2)--(22,-2)--(22,-22)--(2,-22)--(2,-2);
\node at (5,-9) {\normalsize $\bigoplus$};
\node at (7,-9) {\normalsize $\bigoplus$};
\node at (5,-11) {\normalsize $\bigoplus$};
\node at (7,-11) {\normalsize $\bigoplus$};
\node at (5,-13) {\normalsize $\bigoplus$};
\end{tikzpicture}
\caption{The RC-graph $D_{(i,j)}^{(0)}$ and $D_{(i,j)}$ for $w=3,9,2,10,1,8,5,7,4,6$ and $(i,j)=(4,4)$ where important empty space is shaded.}
\label{fig:Dij}
\end{figure}
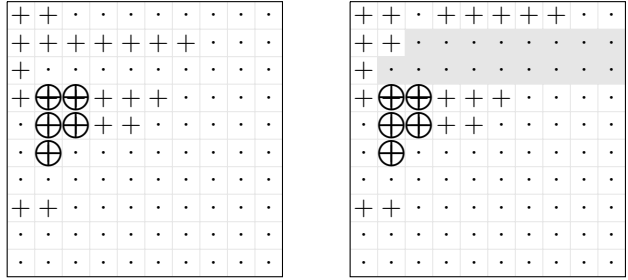

As a summary, we have $\{i,\ldots,i+r\}\times\{l,\ldots,l+m\}\setminus\{(i+r,l+m)\}\subset D_{(i,j)}$, $(i+r,l+m)\notin D_{(i,j)}$ and $(i',j')\notin D_{(i,j)}$ for all $i-\#A_{(i,j)}\leq i'\leq i-1$ and $j'\geq (i+j-1)-i'$. Recall $C_{(i,j)}=\{c_1,\ldots,c_q\}$ and each permutation entry $c_2,\ldots,c_q$ will create a box in row $i$ so we also have $(i,l{+}m{+}1),\ldots,(i,l{+}m{+}q_0)\in D_{(i,j)}$ for some $q_0\geq q-1$ with $(i,l{+}m{+}q_0{+1})\notin D_{(i,j)}$. We can now construct $\P_{(i,j)}\subset\RC(w)$ of cardinality $\#A_{(i,j)}\cdot\#C_{(i,j)}$. Choose any $0\leq a'<\#A_{(i,j)}$ and $0\leq c'<\#C_{(i,j)}=q$. From $D_{(i,j)}$, applying ladder moves of order $r$ to crossings $(i+r,l+m-1),\ldots,(i+r,l)$ in such order to obtain crossings $(i-1,l+m),\ldots,(i-1,l+1)$, applying simple ladder moves to crossings $(i,l{+}m{+}q_0),\ldots,(i,l{+}m{+}q_0{-}c'{+}1)$ to obtain crossings $(i-1,l{+}m{+}q_0{+}1),\ldots,(i-1,l{+}m{+}q_0{-}c'{+}2)$ and applying simple ladder moves $a'$ times to these newly obtained crossings together (from right to left) result in $D_{(i,j)}^{(a',c')}$. Notice that $l{+}m{+}q_0{-}c'{+}2$ is at least 2 greater than $l{+}m$ and that there are at least $\#A_{(i,j)}$ empty rows above them at diagonal number $\geq i+l-1$ so these moves are possible. See Figure~\ref{fig:Dijac} for an example.
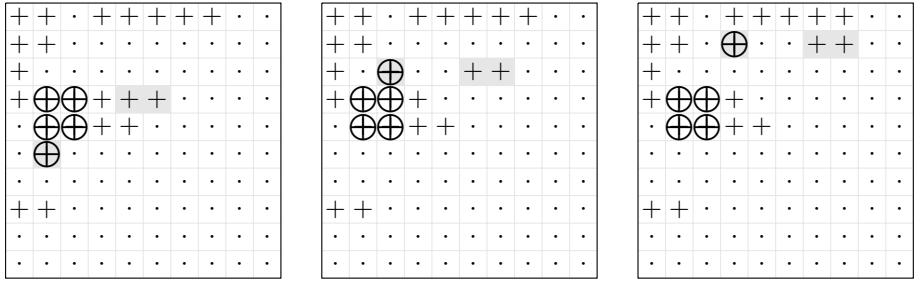
\begin{figure}[h!]
\centering
\begin{tikzpicture}[scale=0.183000000000000]
\fill[gray!20!white] (4,-12) rectangle (6,-14);
\fill[gray!20!white] (10,-8) rectangle (12,-10);
\fill[gray!20!white] (12,-8) rectangle (14,-10);
\node at (3,-3) {\normalsize $+$};
\node at (5,-3) {\normalsize $+$};
\node at (7,-3) {\normalsize $\cdot$};
\node at (9,-3) {\normalsize $+$};
\node at (11,-3) {\normalsize $+$};
\node at (13,-3) {\normalsize $+$};
\node at (15,-3) {\normalsize $+$};
\node at (17,-3) {\normalsize $+$};
\node at (19,-3) {\normalsize $\cdot$};
\node at (21,-3) {\normalsize $\cdot$};
\node at (3,-5) {\normalsize $+$};
\node at (5,-5) {\normalsize $+$};
\node at (7,-5) {\normalsize $\cdot$};
\node at (9,-5) {\normalsize $\cdot$};
\node at (11,-5) {\normalsize $\cdot$};
\node at (13,-5) {\normalsize $\cdot$};
\node at (15,-5) {\normalsize $\cdot$};
\node at (17,-5) {\normalsize $\cdot$};
\node at (19,-5) {\normalsize $\cdot$};
\node at (21,-5) {\normalsize $\cdot$};
\node at (3,-7) {\normalsize $+$};
\node at (5,-7) {\normalsize $\cdot$};
\node at (7,-7) {\normalsize $\cdot$};
\node at (9,-7) {\normalsize $\cdot$};
\node at (11,-7) {\normalsize $\cdot$};
\node at (13,-7) {\normalsize $\cdot$};
\node at (15,-7) {\normalsize $\cdot$};
\node at (17,-7) {\normalsize $\cdot$};
\node at (19,-7) {\normalsize $\cdot$};
\node at (21,-7) {\normalsize $\cdot$};
\node at (3,-9) {\normalsize $+$};
\node at (5,-9) {\normalsize $+$};
\node at (7,-9) {\normalsize $+$};
\node at (9,-9) {\normalsize $+$};
\node at (11,-9) {\normalsize $+$};
\node at (13,-9) {\normalsize $+$};
\node at (15,-9) {\normalsize $\cdot$};
\node at (17,-9) {\normalsize $\cdot$};
\node at (19,-9) {\normalsize $\cdot$};
\node at (21,-9) {\normalsize $\cdot$};
\node at (3,-11) {\normalsize $\cdot$};
\node at (5,-11) {\normalsize $+$};
\node at (7,-11) {\normalsize $+$};
\node at (9,-11) {\normalsize $+$};
\node at (11,-11) {\normalsize $+$};
\node at (13,-11) {\normalsize $\cdot$};
\node at (15,-11) {\normalsize $\cdot$};
\node at (17,-11) {\normalsize $\cdot$};
\node at (19,-11) {\normalsize $\cdot$};
\node at (21,-11) {\normalsize $\cdot$};
\node at (3,-13) {\normalsize $\cdot$};
\node at (5,-13) {\normalsize $+$};
\node at (7,-13) {\normalsize $\cdot$};
\node at (9,-13) {\normalsize $\cdot$};
\node at (11,-13) {\normalsize $\cdot$};
\node at (13,-13) {\normalsize $\cdot$};
\node at (15,-13) {\normalsize $\cdot$};
\node at (17,-13) {\normalsize $\cdot$};
\node at (19,-13) {\normalsize $\cdot$};
\node at (21,-13) {\normalsize $\cdot$};
\node at (3,-15) {\normalsize $\cdot$};
\node at (5,-15) {\normalsize $\cdot$};
\node at (7,-15) {\normalsize $\cdot$};
\node at (9,-15) {\normalsize $\cdot$};
\node at (11,-15) {\normalsize $\cdot$};
\node at (13,-15) {\normalsize $\cdot$};
\node at (15,-15) {\normalsize $\cdot$};
\node at (17,-15) {\normalsize $\cdot$};
\node at (19,-15) {\normalsize $\cdot$};
\node at (21,-15) {\normalsize $\cdot$};
\node at (3,-17) {\normalsize $+$};
\node at (5,-17) {\normalsize $+$};
\node at (7,-17) {\normalsize $\cdot$};
\node at (9,-17) {\normalsize $\cdot$};
\node at (11,-17) {\normalsize $\cdot$};
\node at (13,-17) {\normalsize $\cdot$};
\node at (15,-17) {\normalsize $\cdot$};
\node at (17,-17) {\normalsize $\cdot$};
\node at (19,-17) {\normalsize $\cdot$};
\node at (21,-17) {\normalsize $\cdot$};
\node at (3,-19) {\normalsize $\cdot$};
\node at (5,-19) {\normalsize $\cdot$};
\node at (7,-19) {\normalsize $\cdot$};
\node at (9,-19) {\normalsize $\cdot$};
\node at (11,-19) {\normalsize $\cdot$};
\node at (13,-19) {\normalsize $\cdot$};
\node at (15,-19) {\normalsize $\cdot$};
\node at (17,-19) {\normalsize $\cdot$};
\node at (19,-19) {\normalsize $\cdot$};
\node at (21,-19) {\normalsize $\cdot$};
\node at (3,-21) {\normalsize $\cdot$};
\node at (5,-21) {\normalsize $\cdot$};
\node at (7,-21) {\normalsize $\cdot$};
\node at (9,-21) {\normalsize $\cdot$};
\node at (11,-21) {\normalsize $\cdot$};
\node at (13,-21) {\normalsize $\cdot$};
\node at (15,-21) {\normalsize $\cdot$};
\node at (17,-21) {\normalsize $\cdot$};
\node at (19,-21) {\normalsize $\cdot$};
\node at (21,-21) {\normalsize $\cdot$};
\draw[gray!20!white,ultra thin,step=2] (2,-2) grid (22,-22);
\draw(2,-2)--(22,-2)--(22,-22)--(2,-22)--(2,-2);
\node at (5,-9) {\normalsize $\bigoplus$};
\node at (7,-9) {\normalsize $\bigoplus$};
\node at (5,-11) {\normalsize $\bigoplus$};
\node at (7,-11) {\normalsize $\bigoplus$};
\node at (5,-13) {\normalsize $\bigoplus$};
\end{tikzpicture}
\quad
\begin{tikzpicture}[scale=0.183000000000000]
\fill[gray!20!white] (6,-6) rectangle (8,-8);
\fill[gray!20!white] (12,-6) rectangle (14,-8);
\fill[gray!20!white] (14,-6) rectangle (16,-8);
\node at (3,-3) {\normalsize $+$};
\node at (5,-3) {\normalsize $+$};
\node at (7,-3) {\normalsize $\cdot$};
\node at (9,-3) {\normalsize $+$};
\node at (11,-3) {\normalsize $+$};
\node at (13,-3) {\normalsize $+$};
\node at (15,-3) {\normalsize $+$};
\node at (17,-3) {\normalsize $+$};
\node at (19,-3) {\normalsize $\cdot$};
\node at (21,-3) {\normalsize $\cdot$};
\node at (3,-5) {\normalsize $+$};
\node at (5,-5) {\normalsize $+$};
\node at (7,-5) {\normalsize $\cdot$};
\node at (9,-5) {\normalsize $\cdot$};
\node at (11,-5) {\normalsize $\cdot$};
\node at (13,-5) {\normalsize $\cdot$};
\node at (15,-5) {\normalsize $\cdot$};
\node at (17,-5) {\normalsize $\cdot$};
\node at (19,-5) {\normalsize $\cdot$};
\node at (21,-5) {\normalsize $\cdot$};
\node at (3,-7) {\normalsize $+$};
\node at (5,-7) {\normalsize $\cdot$};
\node at (7,-7) {\normalsize $+$};
\node at (9,-7) {\normalsize $\cdot$};
\node at (11,-7) {\normalsize $\cdot$};
\node at (13,-7) {\normalsize $+$};
\node at (15,-7) {\normalsize $+$};
\node at (17,-7) {\normalsize $\cdot$};
\node at (19,-7) {\normalsize $\cdot$};
\node at (21,-7) {\normalsize $\cdot$};
\node at (3,-9) {\normalsize $+$};
\node at (5,-9) {\normalsize $+$};
\node at (7,-9) {\normalsize $+$};
\node at (9,-9) {\normalsize $+$};
\node at (11,-9) {\normalsize $\cdot$};
\node at (13,-9) {\normalsize $\cdot$};
\node at (15,-9) {\normalsize $\cdot$};
\node at (17,-9) {\normalsize $\cdot$};
\node at (19,-9) {\normalsize $\cdot$};
\node at (21,-9) {\normalsize $\cdot$};
\node at (3,-11) {\normalsize $\cdot$};
\node at (5,-11) {\normalsize $+$};
\node at (7,-11) {\normalsize $+$};
\node at (9,-11) {\normalsize $+$};
\node at (11,-11) {\normalsize $+$};
\node at (13,-11) {\normalsize $\cdot$};
\node at (15,-11) {\normalsize $\cdot$};
\node at (17,-11) {\normalsize $\cdot$};
\node at (19,-11) {\normalsize $\cdot$};
\node at (21,-11) {\normalsize $\cdot$};
\node at (3,-13) {\normalsize $\cdot$};
\node at (5,-13) {\normalsize $\cdot$};
\node at (7,-13) {\normalsize $\cdot$};
\node at (9,-13) {\normalsize $\cdot$};
\node at (11,-13) {\normalsize $\cdot$};
\node at (13,-13) {\normalsize $\cdot$};
\node at (15,-13) {\normalsize $\cdot$};
\node at (17,-13) {\normalsize $\cdot$};
\node at (19,-13) {\normalsize $\cdot$};
\node at (21,-13) {\normalsize $\cdot$};
\node at (3,-15) {\normalsize $\cdot$};
\node at (5,-15) {\normalsize $\cdot$};
\node at (7,-15) {\normalsize $\cdot$};
\node at (9,-15) {\normalsize $\cdot$};
\node at (11,-15) {\normalsize $\cdot$};
\node at (13,-15) {\normalsize $\cdot$};
\node at (15,-15) {\normalsize $\cdot$};
\node at (17,-15) {\normalsize $\cdot$};
\node at (19,-15) {\normalsize $\cdot$};
\node at (21,-15) {\normalsize $\cdot$};
\node at (3,-17) {\normalsize $+$};
\node at (5,-17) {\normalsize $+$};
\node at (7,-17) {\normalsize $\cdot$};
\node at (9,-17) {\normalsize $\cdot$};
\node at (11,-17) {\normalsize $\cdot$};
\node at (13,-17) {\normalsize $\cdot$};
\node at (15,-17) {\normalsize $\cdot$};
\node at (17,-17) {\normalsize $\cdot$};
\node at (19,-17) {\normalsize $\cdot$};
\node at (21,-17) {\normalsize $\cdot$};
\node at (3,-19) {\normalsize $\cdot$};
\node at (5,-19) {\normalsize $\cdot$};
\node at (7,-19) {\normalsize $\cdot$};
\node at (9,-19) {\normalsize $\cdot$};
\node at (11,-19) {\normalsize $\cdot$};
\node at (13,-19) {\normalsize $\cdot$};
\node at (15,-19) {\normalsize $\cdot$};
\node at (17,-19) {\normalsize $\cdot$};
\node at (19,-19) {\normalsize $\cdot$};
\node at (21,-19) {\normalsize $\cdot$};
\node at (3,-21) {\normalsize $\cdot$};
\node at (5,-21) {\normalsize $\cdot$};
\node at (7,-21) {\normalsize $\cdot$};
\node at (9,-21) {\normalsize $\cdot$};
\node at (11,-21) {\normalsize $\cdot$};
\node at (13,-21) {\normalsize $\cdot$};
\node at (15,-21) {\normalsize $\cdot$};
\node at (17,-21) {\normalsize $\cdot$};
\node at (19,-21) {\normalsize $\cdot$};
\node at (21,-21) {\normalsize $\cdot$};
\draw[gray!20!white,ultra thin,step=2] (2,-2) grid (22,-22);
\draw(2,-2)--(22,-2)--(22,-22)--(2,-22)--(2,-2);
\node at (5,-9) {\normalsize $\bigoplus$};
\node at (7,-9) {\normalsize $\bigoplus$};
\node at (5,-11) {\normalsize $\bigoplus$};
\node at (7,-11) {\normalsize $\bigoplus$};
\node at (7,-7) {\normalsize $\bigoplus$};
\end{tikzpicture}
\quad
\begin{tikzpicture}[scale=0.183000000000000]
\fill[gray!20!white] (8,-4) rectangle (10,-6);
\fill[gray!20!white] (14,-4) rectangle (16,-6);
\fill[gray!20!white] (16,-4) rectangle (18,-6);
\node at (3,-3) {\normalsize $+$};
\node at (5,-3) {\normalsize $+$};
\node at (7,-3) {\normalsize $\cdot$};
\node at (9,-3) {\normalsize $+$};
\node at (11,-3) {\normalsize $+$};
\node at (13,-3) {\normalsize $+$};
\node at (15,-3) {\normalsize $+$};
\node at (17,-3) {\normalsize $+$};
\node at (19,-3) {\normalsize $\cdot$};
\node at (21,-3) {\normalsize $\cdot$};
\node at (3,-5) {\normalsize $+$};
\node at (5,-5) {\normalsize $+$};
\node at (7,-5) {\normalsize $\cdot$};
\node at (9,-5) {\normalsize $+$};
\node at (11,-5) {\normalsize $\cdot$};
\node at (13,-5) {\normalsize $\cdot$};
\node at (15,-5) {\normalsize $+$};
\node at (17,-5) {\normalsize $+$};
\node at (19,-5) {\normalsize $\cdot$};
\node at (21,-5) {\normalsize $\cdot$};
\node at (3,-7) {\normalsize $+$};
\node at (5,-7) {\normalsize $\cdot$};
\node at (7,-7) {\normalsize $\cdot$};
\node at (9,-7) {\normalsize $\cdot$};
\node at (11,-7) {\normalsize $\cdot$};
\node at (13,-7) {\normalsize $\cdot$};
\node at (15,-7) {\normalsize $\cdot$};
\node at (17,-7) {\normalsize $\cdot$};
\node at (19,-7) {\normalsize $\cdot$};
\node at (21,-7) {\normalsize $\cdot$};
\node at (3,-9) {\normalsize $+$};
\node at (5,-9) {\normalsize $+$};
\node at (7,-9) {\normalsize $+$};
\node at (9,-9) {\normalsize $+$};
\node at (11,-9) {\normalsize $\cdot$};
\node at (13,-9) {\normalsize $\cdot$};
\node at (15,-9) {\normalsize $\cdot$};
\node at (17,-9) {\normalsize $\cdot$};
\node at (19,-9) {\normalsize $\cdot$};
\node at (21,-9) {\normalsize $\cdot$};
\node at (3,-11) {\normalsize $\cdot$};
\node at (5,-11) {\normalsize $+$};
\node at (7,-11) {\normalsize $+$};
\node at (9,-11) {\normalsize $+$};
\node at (11,-11) {\normalsize $+$};
\node at (13,-11) {\normalsize $\cdot$};
\node at (15,-11) {\normalsize $\cdot$};
\node at (17,-11) {\normalsize $\cdot$};
\node at (19,-11) {\normalsize $\cdot$};
\node at (21,-11) {\normalsize $\cdot$};
\node at (3,-13) {\normalsize $\cdot$};
\node at (5,-13) {\normalsize $\cdot$};
\node at (7,-13) {\normalsize $\cdot$};
\node at (9,-13) {\normalsize $\cdot$};
\node at (11,-13) {\normalsize $\cdot$};
\node at (13,-13) {\normalsize $\cdot$};
\node at (15,-13) {\normalsize $\cdot$};
\node at (17,-13) {\normalsize $\cdot$};
\node at (19,-13) {\normalsize $\cdot$};
\node at (21,-13) {\normalsize $\cdot$};
\node at (3,-15) {\normalsize $\cdot$};
\node at (5,-15) {\normalsize $\cdot$};
\node at (7,-15) {\normalsize $\cdot$};
\node at (9,-15) {\normalsize $\cdot$};
\node at (11,-15) {\normalsize $\cdot$};
\node at (13,-15) {\normalsize $\cdot$};
\node at (15,-15) {\normalsize $\cdot$};
\node at (17,-15) {\normalsize $\cdot$};
\node at (19,-15) {\normalsize $\cdot$};
\node at (21,-15) {\normalsize $\cdot$};
\node at (3,-17) {\normalsize $+$};
\node at (5,-17) {\normalsize $+$};
\node at (7,-17) {\normalsize $\cdot$};
\node at (9,-17) {\normalsize $\cdot$};
\node at (11,-17) {\normalsize $\cdot$};
\node at (13,-17) {\normalsize $\cdot$};
\node at (15,-17) {\normalsize $\cdot$};
\node at (17,-17) {\normalsize $\cdot$};
\node at (19,-17) {\normalsize $\cdot$};
\node at (21,-17) {\normalsize $\cdot$};
\node at (3,-19) {\normalsize $\cdot$};
\node at (5,-19) {\normalsize $\cdot$};
\node at (7,-19) {\normalsize $\cdot$};
\node at (9,-19) {\normalsize $\cdot$};
\node at (11,-19) {\normalsize $\cdot$};
\node at (13,-19) {\normalsize $\cdot$};
\node at (15,-19) {\normalsize $\cdot$};
\node at (17,-19) {\normalsize $\cdot$};
\node at (19,-19) {\normalsize $\cdot$};
\node at (21,-19) {\normalsize $\cdot$};
\node at (3,-21) {\normalsize $\cdot$};
\node at (5,-21) {\normalsize $\cdot$};
\node at (7,-21) {\normalsize $\cdot$};
\node at (9,-21) {\normalsize $\cdot$};
\node at (11,-21) {\normalsize $\cdot$};
\node at (13,-21) {\normalsize $\cdot$};
\node at (15,-21) {\normalsize $\cdot$};
\node at (17,-21) {\normalsize $\cdot$};
\node at (19,-21) {\normalsize $\cdot$};
\node at (21,-21) {\normalsize $\cdot$};
\draw[gray!20!white,ultra thin,step=2] (2,-2) grid (22,-22);
\draw(2,-2)--(22,-2)--(22,-22)--(2,-22)--(2,-2);
\node at (5,-9) {\normalsize $\bigoplus$};
\node at (7,-9) {\normalsize $\bigoplus$};
\node at (5,-11) {\normalsize $\bigoplus$};
\node at (7,-11) {\normalsize $\bigoplus$};
\node at (9,-5) {\normalsize $\bigoplus$};
\end{tikzpicture}
\caption{The RC-graph $D_{(i,j)}$, $D_{(i,j)}^{(0,2)}$ and $D_{(i,j)}^{(1,2)}$ for $w=3,9,2,10,1,8,5,7,4,6$, $(i,j)=(4,4)$ where the crossings moved are shaded.}
\label{fig:Dijac}
\end{figure}

We see that $\#\P_{(i,j)}=\#A_{(i,j)}\cdot\#C_{(i,j)}$ so it remains to show that these RC-graphs are distinct for different $(i,j)$'s. Let $D$ be an RC-graph that is in some $\P_{(i,j)}$ and we will show that we can recover $(i,j)$ from $D$. Compare $D$ with $B_w$. Recall that a crossing has type $(a,b)$ if strand $a$ and strand $b$ intersects at this crossing. Let $X$ be the set of inversions of $w$ such that the crossing in $D$ with type $(a,b)$ has different diagonal number than the crossing in $B_w$ with type $(a,b)$. Notice that simple ladder moves do not change the diagonal number of a crossing of a fixed type. By the construction illustrated above, the crossings in $D$ with types in $X$ must all be in the same row and are consecutive. Let them be $(i',j'),\ldots,(i',j'')$. Let $i$ be the smallest integer such that $i>i'$ and $(i,i'+j'-i)\in D$. Then $(i,j)\in RD(w)$ is the box that corresponds to $\beta(i,j)=(i,i'+j'-i)\in B_w$.
\end{proof}

\section{The main conjecture}\label{sec:conj}
As Theorem~\ref{thm:main} strengthens Theorem~\ref{thm:weigandt}, it is natural to ask how far we can push the results of this form. In this section, we provide our main conjecture which suggests that there is strong relation between $\S_w(1)$ and patterns contained in $w$.

We define a sequence of integers $\{c_u\}_{m\geq1,u\in S_m}$ recursively: $$c_w:=\S_w(1)-1-\sum_{|u|<|w|}c_up_u(w)$$
where $|u|=m$ if $u\in S_m$. 

In other words, we have $\S_w(1)=1+\sum_{|u|\leq|w|}c_up_u(w)$.
\begin{lemma}\label{lem:stab}
For $w\in S_n$ with $w(n)=n$, $c_w=0$.
\end{lemma}
\begin{proof}
We prove the claim by induction. When $w=\id$ is the identity permutation, $\S_{\id}(1)=1$ so $c_{\id}=0$. Now assume that $w(n)=n$ with $w\in S_n$, $n\geq2$ and let $w'\in S_{n-1}$ be the permutation such that $w'(i)=w(i)$ for all $i=1,\ldots,n-1$. The stability property of Schubert polynomials implies that $\S_{w'}(1)=\S_w(1)$. We also observe that for $u\in S_k$ with $k\leq n-1$, $p_u(w')=p_u(w)$ except $u(k)=k$, in which case $c_u=0$ by induction hypothesis. Thus,
\begin{align*}
c_w=&\S_w(1)-1-\sum_{|u|\leq n-1}c_up_u(w)\\
=&\S_{w'}(1)-1-\sum_{|u|\leq n-1}c_up_u(w')\\
=&\S_{w'}(1)-1-\sum_{|u|\leq n-2}c_up_u(w')-\sum_{|u|=n-1}c_up_u(w')\\
=&c_{w'}-c_{w'}p_{w'}(w')=0.
\end{align*}
\end{proof}
Lemma~\ref{lem:stab} allows us to define $c_w$ for $w\in S_{\infty}$. The first few $w\in S_{\infty}$ with nonzero values $c_w$ are $c_{132}=1$, $c_{1432}=1$ followed by 23 permutation patterns in $S_5$.
\begin{conjecture}\label{conj:main}
We have $c_w\geq0$ for all $w\in S_{\infty}$.
\end{conjecture}
Conjecture~\ref{conj:main} has been verified by computer for all $w\in S_n$, $n\leq 8$. Notice that Conjecture~\ref{conj:main} immediately implies Theorem~\ref{thm:main} but only checking $c_{1432}=1$ would not be enough to conclude Theorem~\ref{thm:main}. Computational evidence also seems to suggest that for a fixed $n$, the permutations $w\in S_n$ that achieve maximum are layered. They are listed in Table~\ref{tab:max}.
\begin{table}[h!]
\centering
\begin{tabular}{c|c|c}
$n$ & $\max c_w$ & $w$ \\\hline
3 & 1 & 132 \\
4 & 1 & 1432 \\
5 & 5 & 12543, 21543 \\
6 & 37 & 126543, 216543 \\
7 & 342 & 1327654 \\
8 & 5820 & 13287654
\end{tabular}
\caption{The maximum value of $c_w$ and those permutations $w\in S_n$ that achieve this value for $n\leq 8$.}
\label{tab:max}
\end{table}
This list of permutations is almost identical (except $n=5$) to that of the permutations achieving maximum values at $\S_w(1)$ (\cite{morales2019asymptotics},\cite{stanley2017some}). It is natural to conjecture that these two lists are the same when $n\geq6$.

Certain values of $c_w$'s can be easily computed. First, if $w$ avoids 132 then all of its patterns avoid 132 and it is well-known that $\S_w(1)=1$. So we know $c_w=0$ by induction. Let $w^{(n)}\in S_n$ denote the permutation $1,n,n-1,\ldots,2$. We know that $\S_{w^{(n)}}=C_{n-1}$, the $(n-1)^{th}$ Catalan number \cite{woo2004catalan}. Such $w^{(n)}$ can only contain a permutation $w^{(m)}$, for some $m\leq n$ or $u=m,m-1,\ldots,1$ for which $c_u=0$. A straightforward calculation using induction produces the constants $c_{w^{(n)}}$ for $n\geq3$ and this sequence is recorded as A005043 in OEIS.

\begin{remark}
Conjecture~\ref{conj:main} was observed independently by Christian Gaetz while this paper is in the writing process.
\end{remark}

\section{Simple ladder moves and 1432 avoiding permutations}\label{sec:1432}
In this section, we explore more regarding simple ladder moves on RC-graphs. The main theorem of this section is the following.
\begin{theorem}\label{thm:1432}
The following two conditions for $w\in S_\infty$ are equivalent:
\begin{enumerate}
    \item any two RC-graphs of $w$ are connected by simple ladder moves;
    \item $w$ avoids 1432.
\end{enumerate}
\end{theorem}

\begin{remark}
If $w$ avoids 321, which is a stronger condition than $w$ avoiding 1432, then $w$ is fully commutative: any two reduced expressions for $w$ are connected by commutator moves $s_is_j=s_js_i$ for $|i-j|\geq2$, meaning that there are no Coxeter moves $s_is_{i+1}s_i=s_{i+1}s_is_{i+1}$ available for any reduced expression of $w$. This stronger condition immediately implies (1) of Theorem~\ref{thm:1432} since a non-simple ladder move changes how many simple transpositions $s_i$'s are used (or equivalently, the label of $D$) in the reduced expression $\prod_{i=1}^{\infty}\prod_{j=\infty,(i,j)\in D}^{1}s_{i+j-1}$ of $w$ corresponding to $D\in\RC(w)$.
\end{remark}

We start with some observations on simple ladder moves on RC-graphs. Let $D\in\RC(w)$. Consider all positions on diagonal $k$ and $k+1$ in the following order from lower left to upper right: $(k+1,1),(k,1),(k,2),(k-1,2),\ldots,(1,k),(1,k+1)$. Recall that a simple ladder move at coordinate $(a,b)$ require that $(a,b)\in D$, $(a-1,b),(a,b+1),(a-1,b+1)\notin D$ and it changes $(a,b)$ to $(a-1,b+1)$, staying on the same diagonal. Therefore, for all crossings in $D$ on diagonal $k$ and $k+1$, their relative positions in this ordering will never change after simple ladder moves. Notice that each crossing is labeled uniquely so we can identify the same crossing in different RC-graphs.

\begin{lemma}\label{lem:samerow}
Let $D\in\RC(w)$ be obtained from $B_w$ after a sequence of ladder moves. If $(i,j),(i,j+1)\in D$, then their corresponding crossings (crossings with the same label) in $B_w$ are adjacent in the same row as well.
\end{lemma}
\begin{proof}
Recall that crossings of $B_w$ are left-justified: on row $i$, the crossings in $B_w$ are exactly $(i,1),\ldots,(i,\code(w)_i)$.
Let $k=i+j-1$ so that $\diag(i,j)=k$ and $\diag(i,j+1)=k+1$. Consider the chain of all positions on diagonal $k$ and $k+1$: $(k+1,1),(k,1),\ldots,(1,k),(1,k+1)$. We see that in $D$, $(i,j)$ appears before $(i,j+1)$ and that there are no other crossings in between. Let $(a,b),(a',b')\in B_w$ be the corresponding crossings to $(i,j),(i,j+1)\in D$ respectively. With the observation above, in the chain of all positions on diagonal $k$ and $k+1$, $(a,b)$ must appear before $(a',b')$ and that there are no crossings in between. If $b'=1$, then $a'=k+1$ so $(a',b')$ is the start of this chain and there are no positions available for $(a,b)$. Thus $b'\neq1$. Then $(a',b'-1)\in B_w$ which implies $(a',b'-1)=(a,b)$ so we are done.
\end{proof}

We point out that simple ladder moves satisfy the ``diamond condition" and thus have ``local confluence". Specifically, for $D\in\RC(w)$, if we can apply simple ladder moves at two different crossings $(i_1,j_1)\in D$ and $(i_2,j_2)\in D$, resulting in two RC-graphs $D_1$ and $D_2$ respectively. Then we can apply a simple ladder move at $(i_2,j_2)\in D_1$ and $(i_1,j_1)\in D_2$, resulting in the same RC-graph. This ``diamond property" is easy to verify. The following diamond lemma is very commonly seen in the theory of chip-firing, which is used to show that the resulting stable configuration does not depend on the firing sequence.
\begin{lemma}[Diamond lemma]\label{lem:diamond}
Let $G$ be a directed graph with a unique source $s$ such that for any vertex $v$ with at least two outgoing edges $v\rightarrow v_1$, $v\rightarrow v_2$, there exists some other vertex $v'$ such that $v_1\rightarrow v'$ and $v_2\rightarrow v''$. Then for any vertex $v$, all paths from $s$ to $v$ have the same length. Moreover, either $G$ is infinite and does not have a sink or $G$ has a unique sink.
\end{lemma}
The proof of Lemma~\ref{lem:diamond} is left as an exercise.

In our scenario, $B_w$ is the source that generates a directed graph via simple ladder moves. The unique sink is $T_w$ since it is shown that there exists some sequence of simple ladder moves taking $B_w$ to $T_w$ \cite{weigandt2018schubert}, and that ladder move is applicable to $T_w$ because $T_w$ is top-justified. With the help of the diamond lemma, we immediately know that starting from $B_w$, if we apply any sequence of simple ladder moves until no simple ladder moves are available, then we arrive at $T_w$.

We can rephrase Lemma~\ref{lem:samerow} as follows. Let $D\in\RC(w)$ be obtained from $B_w$ via simple ladder moves, or equivalently, obtained from $T_w$ via inverse simple ladder moves. If $(i,j),(i,j+1)\in D$, then their corresponding boxes in $RD(w)$ lie in the same row and that there are no other boxes in between. Dually, if $(i,j),(i+1,j)\in D$, then their corresponding boxes in $RD(w)$ lie in the same column and that there are no other boxes in between.

\begin{proof}[Proof of Theorem~\ref{thm:1432}]
Recall that simple ladder moves don't change the label of an RC-graph while non-simple ladder moves increase the label lexicographically. We also know that all RC-graphs of $w$ can be obtained from $B_w$ by ladder moves. Therefore, it suffices to show that $w$ contains 1432 if and only if we can apply a non-simple ladder move to some RC-graph of $w$. Therefore, if $w$ contains 1432, the proof of Theorem~\ref{thm:main} produces an RC-graph obtained from $B_w$ via simple ladder moves and at least one non-simple ladder move. This shows (1)$\Rightarrow$(2).

The main case of this theorem is (2)$\Rightarrow$(1). Assume the opposite that $w$ avoids 1432 and that we can apply a non-simple ladder move to an RC-graph $D\in\RC(w)$. It suffices to consider the case that $D$ is obtained from $B_w$ via a sequence of simple ladder moves. Assume that $(i',j'),(i'+1,j'),\ldots,(i'+k,j')\in D$, $(i',j'+1),\ldots,(i'+k-1,j'+1)\in D$ and $(i'+k,j'+1)\notin D$ so that a ladder move of order $k$ can be applied to $(i'+k-1,j'+1)$ in $D$. Here $k\geq1$. For these $2k+1$ crossings in $D$, consider their corresponding boxes in $RD(w)$. By Lemma~\ref{lem:samerow}, there exists $I=\{i_0<\cdots<i_k\}$ and $J=\{j_1<j_2\}$ such that these corresponding boxes are $I\times J\setminus\{(i_k,j_2)\}$. As a basic property of Rothe diagrams, if $(a,b)\in RD(w)$ and $(a',b')\in RD(w)$ with $a<a'$ and $b>b'$, then $(a,b')\in RD(w)$. Together with the fact from Lemma~\ref{lem:samerow} that there are no other boxes in between $(i_a,j_1)$ and $(i_a,j_2)$ for $a=0,\ldots,k$ or $(i_a,j_b)$ and $(i_{a+1},j_{b})$ for $b=1,2$ and $a=0,\ldots,k-1$, we conclude that $I\times J\setminus\{(i_k,j_2)\}$ are the only squares with coordinates in $\{i_0,\ldots,i_k\}\times\{j_1,j_2\}\setminus\{(i_k,j_2)\}$. See Figure~\ref{fig:1432avoid} for an example of these boxes.
\begin{figure}[h!]
\centering
\begin{tikzpicture}[scale=0.200000000000000]
\draw(2,-2)--(16,-2)--(16,-16)--(2,-16)--(2,-2);
\draw[gray!20!white,ultra thin,step=2] (2,-2) grid (16,-16);
\node at (9,-3) {\footnotesize $\bullet$};
\draw(9,-16)--(9,-3)--(16,-3);
\node at (7,-5) {\footnotesize $\bullet$};
\draw(7,-16)--(7,-5)--(16,-5);
\node at (13,-7) {\footnotesize $\bullet$};
\draw(13,-16)--(13,-7)--(16,-7);
\node at (15,-9) {\footnotesize $\bullet$};
\draw(15,-16)--(15,-9)--(16,-9);
\node at (3,-11) {\footnotesize $\bullet$};
\draw(3,-16)--(3,-11)--(16,-11);
\node at (11,-13) {\footnotesize $\bullet$};
\draw(11,-16)--(11,-13)--(16,-13);
\node at (5,-15) {\footnotesize $\bullet$};
\draw(5,-16)--(5,-15)--(16,-15);
\fill[gray!50!white, draw=black] (2,-2) rectangle (4,-4);
\fill[gray!50!white, draw=black] (4,-2) rectangle (6,-4);
\fill[gray!50!white, draw=black] (6,-2) rectangle (8,-4);
\fill[gray!50!white, draw=black] (2,-4) rectangle (4,-6);
\fill[gray!50!white, draw=black] (4,-4) rectangle (6,-6);
\fill[gray!50!white, draw=black] (2,-6) rectangle (4,-8);
\fill[gray!50!white, draw=black] (4,-6) rectangle (6,-8);
\fill[gray!50!white, draw=black] (10,-6) rectangle (12,-8);
\fill[gray!50!white, draw=black] (2,-8) rectangle (4,-10);
\fill[gray!50!white, draw=black] (4,-8) rectangle (6,-10);
\fill[gray!50!white, draw=black] (10,-8) rectangle (12,-10);
\fill[gray!50!white, draw=black] (4,-12) rectangle (6,-14);
\draw(4,-6)--(6,-8);
\draw(6,-6)--(4,-8);
\draw(10,-6)--(12,-8);
\draw(12,-6)--(10,-8);
\draw(4,-8)--(6,-10);
\draw(6,-8)--(4,-10);
\draw(10,-8)--(12,-10);
\draw(12,-8)--(10,-10);
\draw(4,-12)--(6,-14);
\draw(6,-12)--(4,-14);
\end{tikzpicture}
\caption{Rothe diagram of $w=4367152$ which avoids 1432, with $I\times J\setminus\{i_k,j_2\}$ marked.}
\label{fig:1432avoid}
\end{figure}
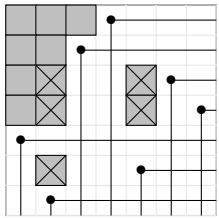

We then claim that $(i_k,j_2)\notin RD(w)$ and in fact $w(i_k)=j_2$. Recall that $(i_{k-1},j_2)\in RD(w)$ corresponds to $(i'+k-1,j'+1)\in D$ and $(i_{k},j_1)\in RD(w)$ corresponds to $(i'+k,j')\in D$, which lie on diagonal $i'{+}j'{+}k{-}1$. Consider the chain of positions on diagonals $i'{+}j{+}k{-}1$ and $i'{+}j{+}k$. Since $(i'+k,j'+1)\notin D$, $(i'+k,j')$ and $(i'+k-1,j'+1)$ are adjacent in this chain. If $(i_k,j_2)\in RD(w)$, then its corresponding crossing in $D$ sits in between, which is a contradiction. As a result $(i_k,j_2)\notin RD(w)$. This means either $j_1<w(i_k)<j_2$, which is impossible since it would imply $(i_0,w(i_k))\in RD(w)$ lying between $(i_0,j_1)$ and $(i_0,j_2)$, or $i_{k-1}<w^{-1}(j_2)<i_k$, which is impossble for the same reason, or $w(i_k)=j_2$, which is only possibility.

As a summary, we see that $(i_0,j_1)\in RD(w)$ and $w(i_k)=j_2$. So $i_0<i_k<w^{-1}(j_1)$ and $w(i_0)>w(i_k)>j_1$. As $w$ avoids $1432$, there does not exist $a<i_0$ such that $w(a)<j_1$. As a result, $(a,b)\in RD(w)$ for all $a\leq i_0$ and $b\leq j_1$. After left-justification to $B_w$, we have $(a,b)\in B_w$ for all $a\leq i_0$ and $b\leq j_1$ as well. These crossings will always be there after any sequence of simple ladder moves. As a result, $(i_0,j_1)\in RD(w)$ corresponds to $(i_0,j_1)\in B_w$ and the same box in $(i',j')\in D$. However, $(i'-1,j')\in D$ as well, contradicting the condition for a ladder move.
\end{proof}

\section*{Acknowledgements}
The author thanks Christian Gaetz, Alex Postnikov and Richard Stanley for helpful conversations. 

\bibliographystyle{plain}

\end{document}